\documentclass[12pt,amscd]{amsart}
\usepackage{graphicx} 
\usepackage{amsmath,amsxtra,amssymb,latexsym, amscd,amsthm}

\usepackage{xcolor}

\DeclareMathOperator{\supp}{supp}
\DeclareMathOperator{\G}{G}
\DeclareMathOperator{\ass}{Ass}
\DeclareMathOperator{\hgt}{ht}
\DeclareMathOperator{\Min}{Min}

\DeclareMathOperator{\lcm}{lcm}
\DeclareMathOperator{\pol}{pol}

\DeclareMathOperator{\depth}{depth}

\newtheorem{thm}{Theorem}[section]

\newtheorem{lem}[thm]{Lemma}

\newtheorem{prop}[thm]{Proposition}
\theoremstyle{definition}

\newtheorem{exm}[thm]{Example}
\newtheorem{rem}[thm]{Remark}

\newtheorem{conj}[thm]{Conjecture}

\def\Z {\mathbb Z}
\def\N {\mathbb N}

\def\p {\mathfrak p}

\title{Cohen-Macaulayness of powers of edge ideals of  edge-weighted graphs}

\author{Jiaxin Li}
\address{School of Mathematical Sciences, Soochow University, Suzhou, Jiangsu, 215006, P.R.~China}
\email{lijiaxinworking@163.com}

\author{Tran Nam Trung}
\address{Institute of Mathematics, Vietnam Academy of Science and Technology, 18 Hoang Quoc Viet, 10307 Hanoi, Vietnam}
\email{tntrung@math.ac.vn}

\author{Guangjun Zhu$^{\ast}$}
\address{School of Mathematical Sciences, Soochow University, Suzhou, Jiangsu, 215006, P.R.~China}
\email{zhuguangjun@suda.edu.cn}

\thanks{$^{\ast}$ Corresponding author}

 \date{\today}

\begin{document}



\begin{abstract}
In this paper, we characterize the Cohen-Macaulayness of  the second power $I(G_\omega)^2$  of the weighted edge ideal $I(G_\omega)$  when the underlying
graph $G$  is a very well-covered graph. We also  characterize  the Cohen-Macaulayness of all ordinary  powers of $I(G_\omega)^n$ when $G$ is a tree with a perfect matching consisting of pendant edges and the induced subgraph  $G[V(G)\setminus S]$ of $G$ on $V(G)\setminus S$ is a star, where $S$ is the set of all leaf vertices, or if  $G$ is a connected graph with a perfect matching consisting of pendant edges and the induced subgraph  $G[V(G)\setminus S]$ of $G$ on $V(G)\setminus S$ is  a complete graph and  the weight  function $\omega$ satisfies $\omega(e)=1$ for all $e\in E(G[V(G)\setminus S])$.

\end{abstract}

\maketitle

\section{Introduction}
Let $G$ be a finite simple  graph with the  vertex set $V(G)$  and  the  edge set $E(G)$.  A subset $X$ of $V(G)$ is said to be {\it independent} if there is no edge $xy \in E(G)$
for $x, y \in X$. If $X$ is an independent set and  maximal in terms of inclusion, then $X$ is called the {\it maximal independent set}  of $G$, and its cardinality
is called the  {\it independence number} of $G$.
A graph $G$ is called {\it well-covered} (also called {\it unmixed}\,) if all maximal independent sets of $G$ have the same independence number. It
 is known that a well-covered graph  is a member
of the class $W_2$ if removing any vertex leaves a well-covered graph with the same independence number as itself (see \cite{P}). A graph $G$ without isolated vertices is said to be {\it very well-covered} if $|V(G)|$ is an even integer and every maximal independent set of $G$ has cardinality $|V (G)|/2$.

 A function $\omega \colon E(G) \to \Z_{>0}$ is called  {\it an edge-weighted function} (weighted function for short) on $E(G)$.  The pair $(G,\omega)$ is called  a {\it weighted graph}  and is denoted by $G_\omega$.
If  $V(G) = \{x_1,\ldots,x_d\}$, then we can assume that  $R = K[x_1,\ldots,x_d]$ is  a polynomial ring over a  field $K$.
The {\it weighted edge ideal} of $G_\omega$ is defined as a monomial ideal of $R$ as follows:
$$I(G_\omega) = ((x_ix_j)^{\omega(x_ix_j)}\mid x_i x_j\in E(G)).$$
If the weighted function $\omega$ on $E(G)$ is the trivial one, i.e. $\omega(e)=1$ for all $e\in E(G)$, then $I(G_\omega)$ is just the classical edge ideal of $G$, denoted by $I(G)$. It is known  (see e.g. \cite[Proposition 6.1.16]{Vi}) that
\[
\ass(R/I(G)) = \{(v:v \in C )\mid  C \text{ is a minimal vertex cover of } G\}.
\]
In particular,  $\dim R/I(G) = \alpha(G)$, where $\alpha(G)$ is the maximum cardinality of independent sets in $G$.
For an ideal $I\subset R$, we say that it is  {\it Cohen-Macaulay {\em (resp. Gorenstein{\em )} } if the quotient ring $R/I$ is Cohen-Macaulay (resp. Gorenstein).
A graph  $G$ is called a {\it Cohen-Macaulay {\em (}resp. Gorenstein{\em )}} graph if its edge ideal  $I(G)$  is Cohen-Macaulay (resp. Gorenstein). It is known that if $G$ is   Gorenstein, then  it is Cohen-Macaulay, so it is also well-covered. Note that $I(G)=\sqrt{I(G_\omega)}$, so $G$ is Cohen-Macaulay if $G_\omega$ is Cohen-Macaulay (due to \cite[Theorem 2.6]{HTT}).

Recall that for a  homogeneous ideal $I\subset R$ and  any integer $n\ge 1$, the $n$-th symbolic power of $I$  is defined  as
$$I^{(n)}=\bigcap\limits_{\p\in \Min(I)}I^nR_{\p}\cap R,$$
where $\Min(I)$ is the set of minimal primes of $I$.

Much research has been done on the characterization of Cohen-Macaulay graphs. Examples include bipartite graphs, chordal graphs, toric graphs, large perimeter graphs, etc. \emph{(}see \cite{HH, HHZ,HMT, HoT, T,VVW,V}\emph{)}.
For the higher powers and higher symbolic power of $I(G)$, Hoang and the second author of this paper  proved in \cite{HoT} that $I(G)^2$ is Cohen-Macaulay if and only if $G$ is a triangle-free Gorenstein graph.  Furthermore, 
 Cowsik and  Nori in \cite{CN} showed that for a homogeneous radical ideal $I\subset R$,  $I^{n}$ is
Cohen-Macaulay for infinitely many $n$ if and only if $I$ is a complete intersection. In particular,  $I(G)^n$ is Cohen-Macaulay for infinitely many $n$ if and only if $I(G)$ is a complete intersection. 
Hoang,  Minh and  the second author of this paper   proved in \cite{HMT2} that $I(G)^{(2)}$ is Cohen-Macaulay if and only if $G$ is a Cohen-Macaulay graph and for all edges $xy$, $G_{xy}$ is Cohen-Macaulay and $\alpha(G_{xy})=\alpha(G)-1$, where $G_{xy}=G\setminus (N(x)\cup N(y))$ is the localization of $G$ at the edge $xy$. In \cite{GNK}, Giancarlo, Naoki and  Ken-ichi showed that   $I(G)^{(n)}$ is  Cohen-Macaulay  for some integer  $n\ge 3$  if and only if
$G$ is a disjoint union of finitely many complete graphs. In this case,  $I(G)^{(n)}$ is  Cohen-Macaulay   for all $n$. We know that if $I(G)^n$ is Cohen-Macaulay, then $I(G)^n=I(G)^{(n)}$.

Recently, there has been a surge of interest in characterizing weight functions for which the edge ideal of a weighted graph is Cohen-Macaulay, and in computing the depth of powers of the edge ideals of some weighted graphs. For example, Paulsen and Sather-Wagstaff in \cite{PS}  classified Cohen-Macaulay weighted graphs $G_\omega$ where the underlying
graph $G$ is a  cycle, a tree, or a  complete graph.  Seyed Fakhari et al. in \cite{SSTY} continued this study, they  classified
Cohen-Macaulay weighted graph $G_\omega$ when $G$ is a very well-covered graph. Recently, Diem et al. in \cite{DMV} gave a  complete  characterization of sequentially Cohen-Macaulay weighted graphs. In \cite{W},  Wei classified all Cohen-Macaulay weighted chordal graphs from a purely graph-theoretic point of view. Hien in \cite{Hi}
classified Cohen-Macaulay weighted graphs $G_\omega$ if $G$ has a girth of at least $5$.
In \cite{ZDCL}, Zhu et al. gave some exact formulas for the depth of the powers of the edge ideal of a weighted star graph and also
gave some lower bounds on the depth of powers of the integrally closed weighted path.

In this paper, we will focus on characterize the Cohen-Macaulayness of  the second power $I(G_\omega)^2$ of the weighted edge ideal $I(G_\omega)$  when the underlying
graph $G$  is a very well-covered graph. We also  characterize  the Cohen-Macaulayness of all ordinary  powers of $I(G_\omega)^n$ when $G$ is a tree with a perfect matching consisting of 
 pendant edges and the induced subgraph  $G[V(G)\setminus S]$ of $G$ on $V(G)\setminus S$ is a star, where $S$ is the set of all leaf vertices, or if $G$ is a connected graph with a perfect matching consisting of pendant edges and the induced subgraph  $G[V(G)\setminus S]$ of $G$ on $V(G)\setminus S$ is  a complete graph and  the weight  function $\omega$ satisfies $\omega(e)=1$ for all $e\in E(G[V(G)\setminus S])$.

The article is organized as follows. In Section \ref{sec:prelim}, we will collect  some of the essential definitions and terminology that will be needed later. In Section \ref{sec:Second},  we  will prove that if $G$  is a Cohen-Macaulay very well-covered graph, then $G$ has at least two leaves. Furthermore, if $G$ has exactly two leaves, then it is bipartite. We also give the necessary and sufficient condition that
$I(G_\omega)^2$ is Cohen-Macaulay using the weight function  $\omega$. In Section \ref{sec:Higher}, if $G$ is a path of length of $3$,  using the weight function  $\omega$,  we  give the necessary and sufficient condition that
$I(G_\omega)^n$ is Cohen-Macaulay for all $n\ge 1$ (see Proposition \ref{n2}). Under the condition that $G$ is a tree, we give the  conditions that the weight function  $\omega$ satisfies if $I(G_\omega)^n$  is Cohen-Macaulay for all $n\ge 1$. Furthermore, if $G$ is a tree with a perfect matching consisting of pendant edges and the induced subgraph  $G[V(G)\setminus S]$ of $G$ on $V(G)\setminus S$ is a star, where $S$ is the set of all leaf vertices, we show that these conditions are crucial.
If $G$ is a connected graph with a perfect matching consisting of pendant edges, and the induced subgraph  $G[V(G)\setminus S]$ of $G$ on $V(G)\setminus S$ is  a complete graph, and  the weight  function $\omega$ satisfies $\omega(e)=1$ for all $e\in E(G[V(G)\setminus S])$, we show that  $I(G_\omega)^n$  is Cohen-Macaulay for all $n\ge 1$ if and only if $\omega(e)\geqslant 2$ for any pendant edge $e$.

\section{Preliminaries}
\label{sec:prelim}
We begin this section by collecting some terminology and results from graph theory. Let $G$ be a finite  simple   graph with a vertex set $V(G)$ and an edge set $E(G)$. A subset $C \subseteq V(G)$ is a {\it vertex cover}  of
$G$ if, for each $e \in E(G)$, $e\cap C \neq \emptyset$.  If $C$ is minimal with respect to inclusion, then $C$ is called a {\it minimal vertex cover} of $G$.  
Obviously,  $C$ is a vertex cover if and only if $V(G)\setminus C$ is an independent set.
If $v$ is a vertex in $G$, its {\it neighborhood} is the set  $N_G(v)=\{u \in V(G)\mid  vu\in E(G)\}$ and its {\it degree}, denoted by $\deg_{G}(v)$, is $|N_G(v)|$.
If $\deg_G(v) = 1$, then $v$ is called a {\it   leaf vertex}. An edge that is incident with a leaf vertex is called a pendant edge.  The {\it closed neighborhood} of $v$ is  $N_G[v] = N_G(v) \cup \{v\}$; if there is no ambiguity on $G$, we use $N(v)$ and $N[v]$, respectively. For an independent set $S$ of $G$, we denote the closed neighborhood of $S$ by $N_G[S] = S \cup \{v\in V(G) \mid N_G(v)\cap S \ne \emptyset\}$; we write $N_G[v_1,\ldots,v_s]$ stands for $N_G[S]$ if $S=\{v_1,\ldots,v_s\}$.

A subset $M\subset E(G)$ is a {\it matching} if $e\cap e'=\emptyset$ for every pair of edges $e, e'\in  M$. If every vertex of $G$ is incident with an edge in $M$, then $M$ is a {\it perfect matching} of $G$. For a subset $A\subset V(G)$,  let $G[A]$ denote the {\it induced subgraph} of $G$ on $A$. At the same time, we denote the induced subgraph of $G$
on $V(G)\setminus A$ by $G\setminus A$.  If $A=\{v\}$, we will write $G\setminus
v$ instead of $G\setminus \{v\}$ for simplicity.
Let $G_\omega$ be a weighted graph, its {\it induced subgraph }  is a weighted graph $H_\omega$, where $V(H)\subset V(G)$,  and for any $u,v\in V(H)$, $uv\in E(H)$  if and only if  $uv\in E(G)$,
 and the weight $\omega_H(uv)$ of edge $uv$ in $H$ is equal to its weight $\omega_G(uv)$ in $G$.

\medskip
A monomial ideal $I$ is {\it unmixed} if its associated primes have the same height.  It is well known that $I$ is unmixed if $R/I$ is Cohen-Macaulay.  If $I(G_\omega)$ is unmixed,  we say that $G_\omega$ is unmixed.  Note that $I(G)=\sqrt{I(G_\omega)}$, so $G$ is well-covered if $G_\omega$ is unmixed.
Note also that the height $I(G_\omega)$ is equal to the smallest size of the  vertex covers of $G$.

\medskip

Next, we recall a formula by Hochster to ompute  the depth of $R/I$, where $I$ is a monomial ideal. Actually, it gives us a way to check the Cohen-Macaulayness of monomial ideals.

\begin{lem}\cite[Theorem 7.1]{HT} Let $I$ be a monomial ideal. Then
$$\depth R/I = \min\{\depth R/\sqrt{I\colon f}\mid f \text{ is a monomial such that } f \notin I\}.$$
\end{lem}

As a consequence, we obtain

\begin{lem} \label{lemCM} Let $I$ be a monomial ideal. Then $I$ is Cohen-Macaulay if and only if $I$ has no embedded prime ideal and $I\colon f$ is Cohen-Macaulay for every monomial $f\notin I$.
\end{lem}

\medskip
  Let $V$ be the set of all variables of $R$. For a monomial $f$ in $R$, its support is  $\supp(f) = \{x\in V \mid  x \text{ is divisible by } f\}$, i.e., it is the set of all variables appearing in $f$.
For a monomial ideal $I$ of $R$ and a subset $W$ of $V$,
the restriction of $I$ on $W$ is
$$I|_W = (f \mid f \text{ is a minimal generator of\ } I \text{ such that } \supp(f)\subseteq W),$$ denoted by $I|_W$,
and define $I_W = IR_W \cap R$ as  the localization of $I$ with respect to $W$. Note that $I_W = I \colon f^\infty$ with $f=\prod_{x\in W} x$.

\medskip
The following lemmas are obvious.
\begin{lem}\label{delete} Let $I$ and $J$ be two monomial ideals in $R$ and let $W$ be a subset of $V$. Then
\begin{enumerate}
    \item $(I\cap J)|_W=I|_W \cap J|_W$; and
    \item $(I^n)|_W = (I|_W)^n$ for all $n\geqslant 1$.
\end{enumerate}
\end{lem}

\begin{lem}\label{localization} Let $I$ and $J$ be two monomial ideals in $R$ and let $W\subseteq V$. Then
\begin{enumerate}
    \item $(I\cap J)_W=I_W \cap J_W$;
    \item $(I^n)_W = (I_W)^n$ for all $n\geqslant 1$; and
    \item if $I$ is Cohen-Macaulay, then  $I_W$ is also Cohen-Macaulay. 
\end{enumerate}
\end{lem}

\begin{lem} \label{L1} Suppose that $G$ has a pendant edge $xy$ with $\deg_G(y)=1$. For any integer $n\ge 1$, if $I(G_\omega)^n$ is unmixed, then $I((G\setminus x)_\omega)^n$ is also unmixed.
\end{lem}

\begin{proof} For simplicity, set $I = I(G_\omega)$ and $H = G\setminus x$. Assume that 
\[
I = Q_1\cap Q_2\cap\cdots\cap Q_{k}\cap Q_{k+1}\cap\cdots\cap Q_m \cap Q_{m+1}\cap \cdots\cap Q_s
\]
is an irredundant primary decomposition of $I$, where
\begin{itemize}
    \item $\hgt(Q_i) = \hgt(I)$ for $i=1,\ldots,m$; and $\hgt(Q_i) > \hgt(I)$ for $i=m+1,\ldots,s$.
    \item $x\in Q_i$ for $i=1,\ldots,k$; and $y\in Q_i$ for $i=k+1,\ldots, m$.
\end{itemize}
Since $I^n$ is unmixed, we have $I^n = I^{(n)}$, where $I^{(n)}$ is the $n$-th symbolic power of $I$. So 
$I^n = Q_1^n\cap\cdots\cap Q_k^n\cap Q_{k+1}^n \cap \cdots\cap Q_{m}^n$.

Let $I'$ be the restriction of $I$ on $V(G)\setminus \{x\}$, then, 
by Lemma \ref{delete}, we have 
$I(H_\omega)^n = (I')^n=(I^n)'$.  So again by this lemma we get that
$$I(H_\omega)^n = {Q'}_1^n\cap\cdots\ {Q'}_k^n\cap Q_{k+1}^n \cap \cdots\cap Q_m^n$$
 is a primary decomposition of $I(H_\omega)^n$.

Since every monomial generator of $I(H_\omega)^n$ contains neither $x$ nor $y$, it follows that
$$I(H_\omega)^n = {Q'}_1^n\cap\cdots\cap {Q'}_k^n.$$
Since every ${Q'}_i$ with $i=1,\ldots,k$ has height $\hgt(I)-1$, this representation is an irredundant primary decomposition of $I(H_\omega)^n$, and thus $I(H_\omega)^n$ is unmixed, as required.
\end{proof}

Further, we obtain
\begin{lem} \label{L2} Suppose that $G$ has a pendant edge $xy$ with $\deg_G(y)=1$. For any integer $n\ge 1$, if  $I(G_\omega)^n$ is Cohen-Macaulay, then $I((G\setminus x)_\omega)^n$ is Cohen-Macaulay.
\end{lem}
\begin{proof} Let $I = I(G_\omega)$,  $H = G\setminus x$ and
$$I = Q_1\cap \cdots\cap Q_{k} \cap Q_{k+1}\cap\cdots\cap Q_m \cap Q_{m+1}\cap \cdots\cap Q_s$$
be an irredundant primary decomposition of $I$ such that 
\begin{itemize}
    \item $\hgt(Q_i) = \hgt(I)$ for $i=1,\ldots,m$; and $\hgt(Q_i) > \hgt(I)$ for $i=m+1,\ldots,s$.
    \item $x\in Q_i$ for $i=1,\ldots,k$; and $y\in Q_i$ for $i=k+1,\ldots, m$.
\end{itemize}
Since $I^n$ is Cohen-Macaulay, it is unmixed. Hence  $I^n = Q_1^n\cap\cdots\cap Q_k^n\cap Q_{k+1}^n \cap \cdots\cap Q_{m}^n$ by Lemma \ref{L1}.

Let $I'$ be the restriction of $I$ on $V(G)\setminus \{x\}$, then, 
by the proof of Lemma \ref{L1}, we get that
$$I(H_\omega)^n = {Q'}_1^n\cap\cdots\cap {Q'}_k^n$$
is an irredundant primary decomposition of $I(H_\omega)^n$.

Let $P_i = \sqrt{Q_i}$ and  $P'_i = \sqrt{Q'_i}$ for $i=1,\ldots,k$. Let $f$ be any monomial with $\supp(f) \subseteq V(G)\setminus \{x,y\}$ such that $f\notin I(H_\omega)^n$. Suppose $f\notin {Q'}_i^n$ for $i=1,\ldots,t$ and $f\in {Q'}_i^n$ for $i=t+1,\ldots,k$. In particular,
$$\sqrt{I(H_\omega)^n \colon f} = \sqrt{{Q'}_1^n\colon f}\cap \cdots\cap \sqrt{{Q'}_t^n\colon f} = {P'}_1\cap \cdots \cap {P'}_t.$$

Let $g = fy^{n\omega(xy)}$, then $g \notin Q_i^n$ for $i=1,\ldots,t$; and $g \in Q_i^n$ for $i=t+1,\ldots,m$. It follows that
$$I(G_\omega)^n : g = (Q_1^n \colon g) \cap \cdots \cap (Q_t^n \colon g),$$    
and so $$\sqrt{I(G_\omega)^n : g} = \sqrt{{Q}_1^n\colon g}\cap \cdots\cap \sqrt{{Q}_t^n\colon g} = {P}_1\cap \cdots \cap {P}_t = (x)+ {P'}_1\cap \cdots \cap {P'}_t.$$

Since $I(G_\omega)^n : g$ is Cohen-Macaulay by Lemma  \ref{lemCM}, so is $ (x)+ {P'}_1\cap \cdots \cap {P'}_t$. Therefore,
$\sqrt{I(H_\omega)^n : f}={P'}_1\cap \cdots \cap {P'}_t$ is Cohen-Macaulay.  It follows that $I(H_\omega)^n$ is Cohen-Macaulay again by Lemma \ref{lemCM}, as required.
\end{proof}

For a monomial $u$ and a variable $x$, we denote $\deg_x(u)=\max\{m: x^m|u\}$.

\begin{lem} \label{L4} Suppose that $I(G_\omega)^n$ is unmixed for some $n\geqslant 1$. If $G$ has a pendant edge  $xy$ with $\deg_G(y)=1$, then $\omega(xy) \geqslant \omega(xz)$ for all $z\in N_G(x)$.
\end{lem}

\begin{proof} Let $I = I(G_\omega)$ and let $I =Q_1\cap \cdots\cap Q_{k} \cap Q_{k+1}\cap\cdots\cap Q_m \cap Q_{m+1}\cap \cdots\cap Q_s$ be an irredundant primary decomposition of $I$, where $\hgt(Q_i)=\hgt(I)$ for $i=1,\ldots,m$; and $\hgt(Q_i) > \hgt(I)$ for $i=m+1,\ldots, s$. Since $I^n$ is unmixed, we have
$$I^n = Q_1^n \cap \cdots \cap Q_m^n.$$

Let $P_i = \sqrt{Q_i}$ for $i=1,\ldots,m$, where $x\in P_i$ for $i=1,\ldots,k$, and $x\notin P_i$ for $i=k+1,\ldots,m$. Since $\ass(R/I^n) = \{P_1,\ldots,P_m\}$, it follows that $\ass(R/I^n) = \ass(R/I(G))$. Since $xy$ is a pendant edge of $G$, every associated prime ideal of $G$ contains either $x$ or $y$, but not both. Thus, $y\notin P_i$ for $i=1,\ldots,k$, and $y\in P_i$ for $i=k+1,\ldots,m$. 

For a monomial ideal $J$ of $R$, let $\G(J)$ be the set of its  minimal monomial generators. We have
$$I^n = (Q_1^n\cap\cdots\cap Q_k^n) \cap (Q_{k+1}^n\cap\cdots\cap Q_m^n)= ((x^p)+J_1) \cap ((y^q) + J_2)$$
where each monomial $f\in \G(J_1)$ has $\deg_x(f) < p$ and $f \nmid x^p$; and each monomial $g\in \G(J_2)$ has $\deg_y(g) < q$ and $g \nmid y^q$. Since $(xy)^{n\omega(xy)}$ is a monomial generator of $I^n$, it forces
$(xy)^{n\omega(xy)} = x^py^q$, and so $p = q= n\omega(xy)$.

Now for every $z\in N_G(x)\setminus\{y\}$. It is easy to check that $(xy)^{(n-1)\omega(xy)}(xz)^{\omega(xz)}$ is a monomial generator of $I^n$. It follows that 
$(xy)^{(n-1)\omega(xy)}(xz)^{\omega(xz)} = \lcm(f,g)$ for some $f\in \G((x^p)+J_1)$ and $g\in \G((y^q)+J_2)$. Note that $g$ does not contain  $x$, so 
$$\deg_x((xy)^{(n-1)\omega(xy)}(xz)^{\omega(xz)}) = \deg_x(\lcm(f,g)) = \deg_x(f)< p.$$
Thus $(n-1)\omega(xy)+\omega(xz)<n\omega(xy)$, and thus $\omega(xz)\leqslant \omega(xy)$, as required.
\end{proof}

If a graph $G$ has a perfect matching consisting of pendant edges, and $\omega$ is a weight function on $G$, then \cite[Lemma 5.3]{PS} gives a sufficient condition on $\omega$ for  $I(G_\omega)$ to be Cohen-Macaulay. We will generalize this for powers of $I(G_\omega)$.

For this we introduce the polarization of a monomial ideal (see \cite[Chapter I]{PV}). Let $m = x_1^{\alpha_1}x_2^{\alpha_2}\cdots x_d^{\alpha_d}$ in $R$. Let $S = K[x^i_j\mid 1\leqslant i\leqslant d, 1\leqslant j\leqslant N]$ for some $N\gg 0$. Then we associate to $m$ a squarefree monomial $m^{\pol}$ in $S$ as follows:
$$m^{\pol} = x^1_1x^1_2\cdots x^1_{\alpha_1} x^2_1x^2_2\cdots  x^2_{\alpha_2}\cdots x^d_1x^d_2\cdots x^d_{\alpha_d},$$
where each power $x_i^{\alpha_i}$ of a variable $x_i$ is replaced by the product $x^i_1x^i_2\cdots x^i_{\alpha_i}$. We call $m^{\pol}$ the polarization of $m$. Let $I=(m_1,\ldots,m_k)$ be a monomial ideal of $R$, where $\{m_1, \ldots , m_k\}$ is the set of minimal  monomial generators of $I$. Then we define the ideal $I^{\pol}$ of $S$ by $I^{\pol} = (m_1^{\pol}, \ldots, m_k^{\pol})$, which is called the polarization of $I$. The following lemma is an useful fact used in the paper.

\begin{lem}\label{lemPol} $I$ is Cohen-Macaulay if and only if $I^{\pol}$ is Cohen-Macaulay.
\end{lem}
\begin{proof}
    It follows from \cite[Theorem 21.10]{PV}.
\end{proof}

\begin{thm}\label{Pn} Suppose that  $G_\omega$ is a weighted graph such that its underlying graph $G$ has a perfect matching $x_1y_1, \ldots,x_ty_t$, where each $y_i$ is a leaf vertex. If there exists an integer $k\geqslant 1$ such that $\omega(x_iy_i)\geqslant k\omega(x_ix_j)$ for all $i=1,\ldots,t$ and $x_j\in N_G(x_i)$, then $I(G_\omega)^n$ is Cohen-Macaulay for all $n=1,\ldots,k$.
\end{thm}
\begin{proof} For simplicity, we set $I = I(G_\omega)$, $m_i = \omega(x_iy_i)$ for all $i=1,\ldots,t$, and $\omega_{ij}=\omega(x_ix_j)$ for all $x_ix_j\in E(G)$. Then
$$I=(\{(x_iy_i)^{m_i}\mid i=1,\ldots,t\}\cup \{(x_ix_j)^{\omega_{ij}}\mid x_ix_j\in E(G)\}).
$$

Let $J$ be a monomial ideal in the polynomial  ring $K[u_1,\ldots,u_t]$ defined by
$$J=(\{u_i^{2m_i}\mid i=1,\ldots,t\}\cup \{(u_iu_j)^{\omega_{ij}}\mid x_ix_j\in E(G)\}),$$
i.e., $J$ is obtained from $I$ by replacing $y_i$ and $x_i$ by $u_i$ for all $i$. Note that $J$ is Artinian.

Fix an arbitrary integer $n=1,\ldots,k$. To prove that $I^n$ is Cohen-Macaulay, we use the technique of polarization. Let $S$ be the polynomial ring
$$S = K[x^1_1,\ldots,x^1_{2nm_1},\ldots, x^t_1,\ldots,x^t_{2nm_t}]$$
of variables $x^i_j$ for $i=1,\ldots,t$ and $j=1,\ldots,2nm_i$.

Then, for $1\leqslant m\leqslant nm_i$, we define the polarizations of $x_i^m$ and $y_i^m$ in $R$ to be monomials in $S$  as follows. Let $m = m_is+r$ where $0\leqslant r <  m_i$. Define
$$(x_i^m)^{\pol} = (x^i_{2sm_i+1}\cdots x^i_{2sm_i+r})\prod_{p=0}^{s-1} \prod_{j=1}^{m_i}x^i_{2pm_i+j},$$
and
$$(y_i^m)^{\pol} = (x^i_{(2s+1)m_i+1}\cdots x^i_{(2s+1)m_i+r})\prod_{p=0}^{s-1} \prod_{j=1}^{m_i}x^i_{(2p+1)m_i+j}.$$

Next, we define the polarization of $u_i^m$ with $1\leqslant m \leqslant 2nm_i$ to be a monomial in $S$ by
$$(u_i^m)^{\pol} = x^i_1\cdots x^i_m.$$

Using the condition that $\omega(x_iy_i) \geqslant k\omega(x_ix_j)$ for each $i=1,\ldots, t$ and $x_j\in N_G(x_i)$, we can verify the following claim: Let $f$ be a product of $n$ monomial generators of $I$ and let $g$ be a monomial of $K[u_1,\ldots,u_t]$ obtained from $f$ by replacing $x_i$ and $y_i$ by $u_i$  for each $i=1,\ldots,t$. Then $f^{\pol} = g^{\pol}$.

This claim yields $(I^n)^{\pol} = (J^n)^{\pol}$. Since $J$ is Artinian, so is $J^n$, and so  $J^n$ is Cohen-Macaulay. So $(J^n)^{\pol}$ is Cohen-Macaulay and  so $I^n$ is also Cohen-Macaulay by Lemma \ref{lemPol}, as required.
\end{proof}

\section{Second powers of edge ideals of weighted bipartite graphs}
\label{sec:Second}

In this section, we will characterize the Cohen-Macaulayness of $I(G_\omega)^2$, where $G_\omega$ is a weighted very well-covered graph. Note that if $I(G_\omega)^2$ is Cohen-Macaulay, then $I(G)=\sqrt{I(G_\omega)^2}$  is also Cohen-Macaulay by \cite[Theorem 2.6]{HTT}. 

Thus, in this section we will  always consider $G$ to be a Cohen-Macaulay very well-covered graph.  By \cite[Theorem 3.6]{CRT} we can partition $V(G)$ into $V(G) = X \cup Y$ where $X$ is a minimal vertex cover and $Y$ is a maximal independent set  such that  the vertices of  $G$ are labeled such that the following condition  holds.

$(*)$  $X=\{x_1,\ldots,x_t\}$ and $Y=\{y_1,\ldots,y_t\}$ 

such that
\begin{enumerate}
    \item [$(1*)$] $x_iy_i\in E(G)$ for all $i=1,\ldots,t$;
    \item [$(2*)$] if $x_iy_j\in E(G)$, then $i \leqslant j$; and
    \item [$(3*)$] if $x_iy_j\in E(G)$, then $x_i x_j \notin E(G)$; and
    \item [$(4*)$] if $x_iy_k, x_ky_j\in E(G)$ for distinct indices $i,j,k$, then $x_iy_j\in E(G)$.
    \item [$(5*)$] if $x_iy_k, x_kx_j\in E(G)$ for distinct indices $i,j,k$, then $x_ix_j\in E(G)$.
\end{enumerate}

If $G$ is a bipartite graph, we can assume that $Y$ is a maximal independent set of $G$ such that $(X,Y)$ is a bipartition of $G$ (see \cite{HH}). An example of such a graph is shown in  Figure \ref{wbg}, where we set $\omega_{ij} = \omega(x_iy_j)$ for all $x_iy_j\in E(G)$.
\begin{figure}[ht]
\includegraphics[scale=0.45]{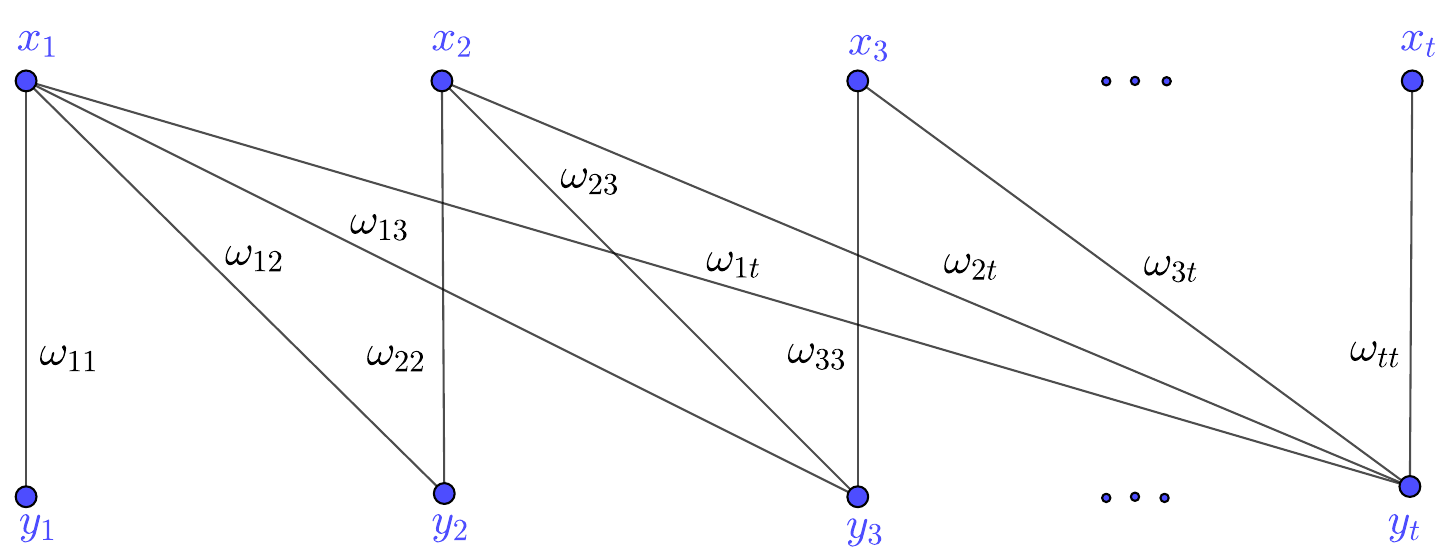}\\
\medskip
\caption{The weighted bipartite graph.}
\label{wbg}
\end{figure}

\begin{lem}\label{cmCover} Every associated prime ideal of $I(G)$ contains either $x_k$ or $y_k$, but not both, for every $k=1,\ldots,t$.
\end{lem}
\begin{proof} Let $\p$ be an associated prime ideal of $I(G)$ such that $\p = (C)$ for some minimal cover $C$ of $G$. Let $T= V(G)\setminus C$, then  $T$ is a maximal independent set of $G$. If $x_k,y_k\notin T$,  we have two cases:

{\it Case $1$}: $x_iy_k\in E(G)$  and  $x_kx_j\in E(G)$ for some $x_i,x_j\in T$. By Property $(5*)$, we have $x_ix_j\in E(G)$,  a contradiction.

{\it Case $2$}: $x_iy_k\in E(G)$  and  $x_ky_j\in E(G)$ for some $x_i,y_j\in T$.  By Property $(4*)$, we have $x_iy_j\in E(G)$,  a contradiction.

So either $x_k$ or $y_k$ is in $T$, but not both, since $T$ is an independent set. This shows that $C$ contains either $x_k$ or $y_k$, but not both, and the lemma follows.
\end{proof}

\begin{lem}\label{leaf2} $G$ has at least two leaves.
\end{lem}
\begin{proof} The case $t=1$ is obvious, so  we assume $t\geqslant 2$. Since $y_1$ is a leaf, it remains to show that $G$ has another leaf. We consider the following two cases:

(1) If there exists some  $y_k$ such that $x_1y_k \notin E(G)$, then we take $k$ to be a minimal index among such indices. In this case, we claim that $y_k$ is a leaf, so that   $y_1$ and $y_k$ are desired leaves.

In fact, if  $y_k$ is not a leaf, then $x_iy_k\in E(G)$ for some index $1<i< k$. It follows that $x_1y_i\in E(G)$ by the minimality of $k$, so $x_1y_k\in E(G)$ according to Property $(4*)$, which is a contradiction.

(2) If $x_1y_k\in E(G)$ for all $k=1,\ldots,t$, then, by Property $(3*)$, we get $x_1x_k\notin E(G)$ for all $k=2,\ldots,t$.
In this case, we claim that: $x_ix_j\notin E(G)$ for all $2\le i<j\le t$. 
Indeed, if there exist some $2\le i<j\le t$ such that $x_ix_j\in E(G)$. Since $x_1y_i\in E(G)$, we can obtain that $x_1x_j\in E(G)$ by  Property $(5*)$, which contradicts Property $(3*)$. So $x_t$ is a leaf, as required.
\end{proof}

\begin{lem}\label{bp} If $G$ has exactly two leaves, then it is bipartite.
\end{lem}
\begin{proof} The case where $t = 1$ is trivial, since in this case $G$ has exactly one edge. If $t=2$, then $G$ is a path of length $3$, and it is obvious that it is bipartite.
Suppose $t\geqslant 3$. It follows  from Property $(2*)$ that $y_1$ is a leaf.  Consider the following two cases:

(1) If $y_j$ is not a leaf of $G$ for all  $j\geqslant 2$, then $x_1y_j\in E(G)$ for all $2\le j\le t$. Indeed,
 if $x_1y_j\notin E(G)$ for some $2\le j\le t$, then we see  from the proof of Lemma \ref{leaf2} that $y_j$ is a leaf of $G$,   a contradiction.

Now we  show that in this case  $X$ is an independent set. Suppose  that this is not the case. Then there exist  $p,q$ with $1\le p<q\le t$ such that  $x_px_q\in E(G)$, so $x_1x_q\in E(G)$  by Property $(5*)$, since $x_1y_p\in E(G)$. This contradicts Property $(3*)$, so $X$ is an independent set of $G$. This implies that $G$ is bipartite.

\medskip
(2) Suppose there exists some $m$ with $2\leqslant  m\leqslant t$ such that $y_m$ is a leaf in $G$, then the following claim holds:

{\it Claim $1$}:  $x_1y_j\in E(G)$ or $x_my_j\in E(G)$ for all $j\notin \{1,m\}$.  If this is not the case, let $j$ be the minimum index  such that $x_1y_j,x_my_j \notin E(G)$. Since
both $y_1$ and $y_m$ are leaves and $G$ has only two leaves,   $y_j$ is not a leaf. Therefore, by Property $(2*)$,
 there exists some $x_iy_j\in E(G)$ with $i<j$.
 By the minimality of $j$, $x_1y_i\in E(G)$ or  $x_my_i\in E(G)$. We can assume that $x_1y_i\in E(G)$, but then $x_1y_j\in E(G)$ by Property $(4*)$, a contradiction.

We consider the following two subcases:

(i) If $x_1x_m\notin E(G)$, then we have $G\setminus N_G[x_1,x_m]$ is an induced subgraph of $G[V(G)\setminus \{x_1,x_m\}]$. On the other hand, since $G$ is well-covered, by \cite[Lemma 1]{FHN} we have
$$\alpha(G\setminus N_G[x_1,x_m]) = t-2=|X|-2.$$
It forces $X\setminus \{x_1,x_m\}$ to be an independent set of $G$ and $N_G[x_1,x_m] \cap (X\setminus \{x_1,x_m\}) = \emptyset$. This forces $X$ to be an independent set of $G$, so $G$ is bipartite.

\medskip
(ii) If $x_1x_m\in E(G)$. First, we will  prove the  following two claims:

{\it Claim $2$}: $x_1y_j\notin E(G)$ or  $x_my_j\notin E(G)$  for all $j\notin \{1,m\}$. On the contrary, assume that  there exists some $j\notin \{1,m\}$ such that $x_1y_j, x_my_j\in E(G)$. Let $j$ be the maximum index among such indices. Since $x_j$ is not a leaf, we have either $x_jy_q\in E(G)$ for $j<q$ or $x_ix_j\in E(G)$ for $i\ne j$. In the first case,  by Property $(4*)$, we have $x_1y_q,x_my_q\in E(G)$, which contradicts the maximality of $j$.  Hence $x_ix_j\in E(G)$ for some $i\ne j$. By Property $(3*)$, one has $i\notin \{1,m\}$. On the other hand, either $x_1y_i\in E(G)$ or $x_my_i\in E(G)$  by  Claim $1$, so $x_1x_j\in E(G)$ or $x_mx_j\in E(G)$ by Property $(5*)$. This contradicts Property $(3*)$.  Hence $x_1y_j\notin E(G)$ or  $x_my_j\notin E(G)$  for all $j\notin \{1,m\}$.

{\it Claim $3$}:  $x_1x_j\in E(G)$  or  $x_mx_j\in E(G)$ for all $j\notin \{1,m\}$, but  not both. In fact, this is not the case,  let $j$ be the maximal index such that $x_1x_j,x_jx_m \notin E(G)$. Since $x_j$ is not a leaf, we have either $x_jy_k\in E(G)$ for some  $j<k$, or $x_jx_s\in E(G)$ for  $s\notin \{1,j,m\}$. In the first case, we can obtain that $x_1y_k\in E(G)$ or $x_my_k\in E(G)$   by  Claim $1$.  If $x_1y_k\in E(G)$, then $x_1x_k\notin E(G)$ by Property $(3*)$. So $x_kx_m\in E(G)$ by the maximality of $j$. It follows that  $x_jx_m\in E(G)$ by Property $(5*)$.
 If $x_my_k\in E(G)$, then  $x_1x_j\in E(G)$ by similar arguments as in the case $x_1y_k\in E(G)$. These two cases contradict the hypothesis that
$x_1x_j,x_jx_m \notin E(G)$.
For the second case,  by similar arguments as for the first case, we can get that then $x_1x_j\in E(G)$ if $x_1y_s\in E(G)$, and $x_mx_j\in E(G)$
if   $x_my_s\in E(G)$. This also contradicts the hypothesis that $x_1x_j,x_jx_m \notin E(G)$. Therefore, $x_1x_j\in E(G)$  or  $x_mx_j\in E(G)$ for all $j\notin \{1,m\}$.
Since $x_1y_j\in E(G)$ or $x_my_j\in E(G)$  for all $j\notin \{1,m\}$, $x_1x_j\in E(G)$ and  $x_mx_j\in E(G)$ do not simultaneously hold by Property $(3*)$, as claimed.
  
Now we return  to the proof of the lemma. Let $A=\{v\mid vx_1\in E(G)\}$ and $B = \{v\mid vx_m\in E(G)\}$. By Claims $1$, $2$ and $3$, we conclude that $V(G) = A \cup B$ and $A \cap B = \emptyset$. It remains to show that $A$ and $B$ are independent sets of $G$. If $A$ is not an independent set of $G$, then we have two possible cases:

{\it Case $1$}: If $x_ix_j\in E(G)$ for some $x_i,x_j\in A$, then $x_i,x_j\notin B$, since $A \cap B=\emptyset$. It follows that $i\ne m$ and $j\ne m$. Indeed, if $i=m$, then $x_mx_j\in E(G)$.  By the definition of $B$, $x_j\in B$, so $x_j\in A\cap B$, a contradiction.
Similarly, $j\ne m$. On the other hand, since $x_1x_i,x_1x_j\in E(G)$, we have $x_1y_i,x_1y_j\notin E(G)$ by  Property $(3*)$. This forces $x_my_i,x_my_j\in E(G)$ by  Claims $1$ and $2$. Again, using $x_ix_j\in E(G)$ and $x_my_j\in E(G)$, we can see that $x_ix_m\in E(G)$. It follows that $x_i\in A\cap B$, which is impossible.

{\it Case $2$}: If $x_iy_j\in E(G)$ for some $x_i,y_j\in A$, then $y_j\notin B$ since $A \cap B=\emptyset$.  By the definition of $B$, $i\ne m$. Since $x_i,y_j\in A$, $x_1x_i\in E(G)$, $x_1y_j\in E(G)$. It follows from Property $(3*)$ that $i\ne j$.
On the other hand, since $x_1x_i\in E(G)$, we have $x_1y_i\notin E(G)$ by  Property $(3*)$. This forces $x_my_i\in E(G)$ by  Claims $1$ and $2$. Again, using $x_iy_j\in E(G)$, we can see that $x_my_j\in E(G)$. It follows that $y_j\in A\cap B$, which is impossible.

So two cases are impossible. This means that $A$ is an independent set of $G$. Similarly, we can prove that $B$ is also an independent set of $G$. Hence, $G$ is a bipartite graph, as required.
\end{proof}

We now characterize the weight function $\omega$ such that $I(G_\omega)^2$ are Cohen-Macaulay. If $t=1$, then $G$ has only one edge, and any weight function is possible. If $t = 2$, we get the following result.

\begin{lem} \label{twotimes} Let $G_\omega$ be a weighted path  of length $3$ with $E(G) = \{ax,ab,by\}$. If $I(G_\omega)^n$ is Cohen-Macaulay for some $n\geqslant 2$, then $\min\{\omega(ax), \omega(by)\} \geqslant 2 \omega(ab)$.
\end{lem}
\begin{proof} For simplicity, let $I = I(G_\omega)$ and let $p = \omega(ax)$, $k = \omega(ab)$, $q = \omega(by)$ (see Figure \ref{pqTree}). Then 
$$I(G_\omega) = ((ax)^p,(ab)^k,(by)^q)$$
and we need to prove $2k\leqslant\min\{p,q\}$.

First, since $I^n$ is unmixed, we have $k \leqslant \min\{p,q\}$ by Lemma \ref{L4}.

\begin{figure}[ht]
\includegraphics[scale=0.5]{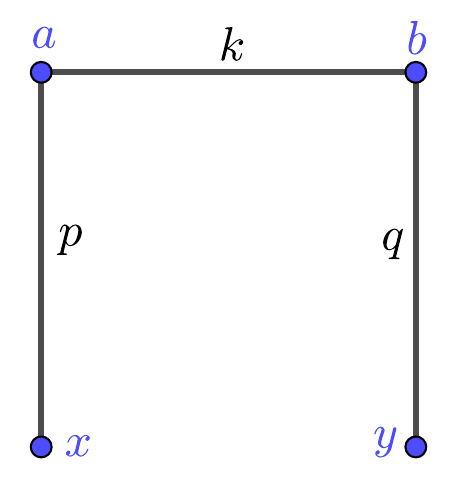}\\
\medskip
\caption{The weighted path of length $3$.}
\label{pqTree}
\end{figure}    

We can assume that $p \leqslant q$, so that $k \leqslant p\leqslant q$. First,  We  claim that $k < q$. Assume that this is not the case, so that $p = k = q$. Hence,
$$I = ((ab)^k,(ax)^k,(by)^k)= (a^k,b^k)\cap (a^k,y^k) \cap (b^k,x^k).$$

Since $I^n$ is unmixed, we have
$$I^n = (a^k,b^k)^n\cap (a^k,y^k)^n \cap (b^k,x^k)^n.$$
Let $f = a^kb^{(n-1)k}$. Using the fact that $n\geqslant 2$, we have $f\in (a^k,b^k)^n$, $f\notin (a^k,y^k)^n$ and $f\notin (b^k,x^k)^n$. It follows that
$$\sqrt{I^n \colon f} = (a,y) \cap (b,x).$$
Note that $\sqrt{I^n \colon f}$ is not Cohen-Macaulay. This contradicts Lemma \ref{lemCM}. Therefore  we have  $k < q$.

We now proceed to prove the lemma. It suffices to show that $p \geqslant 2k$. Suppose, on the contrary that, that $p < 2k$. Since $I^n$ is Cohen-Macaulay,  the localization $I^n_{\{y\}}$ of $I^n$ with respect to $\{y\}$ is also Cohen-Macaulay   by Lemma \ref{localization}. 
 In particular, $I^n_{\{y\}} = I^{(n)}_{\{y\}}$, where $I^{(n)}_{\{y\}}$ is the $n$-th symbolic power of $I_{\{y\}}$. Note that $I_{\{y\}} = (a^px^p, a^kb^k,b^q) = (a^p, a^kb^k,b^q) \cap (b^k, x^p)$, so
$$I^{(n)}_{\{y\}} = (a^p, a^kb^k,b^q)^n \cap (b^k, x^p)^n.$$

Let $f = (a^kb^k)^{n-2}a^pb^s$ where $s = \max\{q,2k\}$, then $f \in (a^p, a^kb^k,b^q)^n \cap (b^k, x^p)^n$. Thus $f\in I^{(n)}_{\{y\}}$ . 

We claim that $f\notin I^n_{\{y\}}$. Indeed, suppose on the contrary that $f\in I^n_{\{y\}}$. Then $f=g(a^kb^k)^{n-u} b^{qu}$, where $g$ is a monomial in $R$ and $0\leqslant u\leqslant n$. It follows that
$$(n-2)k + p \geqslant (n-u)k \text{ and } (n-2)k+s\geqslant (n-u)k + qu,$$
and so $(u-2)k + p \geqslant 0 \text{ and } (u-2)k+s\geqslant qu$.

Since $p < 2k$, the inequality $(u-2)k + p \geqslant 0$ yields $u\geqslant 1$. Note that $k< q$, we obtain
\begin{align*}
q &\leqslant (u-2)k +s -(u-1)q \leqslant(u-2)k+s-(u-1)k=s-k\\
&= \max\{2k-k, q-k\} =\max\{k,q-k\}.
\end{align*}
This is impossible because $1\leqslant k < q$. Hence $f \notin I^n_{\{y\}}$, and hence $I^n_{\{y\}} \ne I^{(n)}_{\{y\}}$, a contradiction. Therefore, $p \geqslant 2k$, and the proof is complete.
\end{proof}

\medskip

\begin{lem} \label{reduceEdges} If  $I(G_\omega)^2$ is unmixed, then  $I((G\setminus \{x_k,y_k\})_\omega)^2$ is also unmixed for all $k=1,\ldots,t$.   
\end{lem}
\begin{proof} Let $I = I(G_\omega)$ and $J=I((G\setminus \{x_k,y_k\})_\omega)$. Since $I^2$ is unmixed, it has the irredundant primary decomposition as
$$ I^2 = Q_1^2\cap Q_2^2\cap\cdots\cap Q_s^2,$$
where each $Q_i$ is a primary monomial ideal and $\p_i = \sqrt{Q_i}$ is the associated prime ideal of $I(G)$. 

Let $W = V(G)\setminus \{x_k,y_k\}$, then  $J = I|_W$ is the restriction of $I$ on $W$. It follows that
\begin{equation} \label{J2EQ}
J^2 = (Q_1|_W)^2\cap (Q_2|_W)^2\cap\cdots\cap (Q_s|_W)^2.    
\end{equation}

On the other hand, for each $i$, by Lemma \ref{cmCover}, $\p_i$ contains either $x_k$ or $y_k$, but not both. It follows that $\hgt(Q_i|_W) = \hgt(Q_i)-1$. Combining this fact with Eq. $(\ref{J2EQ})$, we conclude that $J^2$ is unmixed, as required.
\end{proof}

\begin{lem}\label{TwoQ} Suppose further that $X$ is an independent set of $G$. If  $I(G_\omega)^2$ is Cohen-Macaulay, then the following results hold:
\begin{enumerate}
    \item If  $x_iy_j\in E(G)$ with $i<j$, then $2\omega(x_iy_j) \leqslant \min\{\omega(x_iy_i), \omega(x_jy_j)\}$.
    \item If $x_iy_k$, $x_ky_j\in E(G)$ with  $i<k<j$, then $2\omega(x_iy_j)\leqslant \min\{\omega(x_iy_k), \omega(x_ky_j)\}$.
\end{enumerate}
\end{lem}
\begin{proof} By our assumption, $G$ is a bipartite graph with bipartition $(X,Y)$. We prove the statement by induction on $t$. If $t = 2$, then it follows from Lemma \ref{twotimes}.

Suppose $t\geqslant 3$.  If $i  \ne 1$, let $H = G\setminus x_1$. Since $x_1y_1$ is a pendant edge with $\deg_G(y_1)=1$, we have $I(H_\omega)^2$ is Cohen-Macaulay by Lemma \ref{L2}. Applying the induction hypothesis to $H_\omega$, we obtain  the desired inequalities. Similarly, if $j\ne t$, then the desired inequalities follow by considering the weighted graph $(G\setminus y_t)_\omega$. 

This argument also shows that all the inequalities hold if either $i\ne 1$ or $j \ne t$. Let $\omega_{ij} = \omega(x_iy_j)$ for all $x_iy_j\in E(G)$. It remains to prove the case $i=1$ and $j=t$, i.e., we will prove that 
\begin{equation}\label{PR1}
2\omega_{1t} \leqslant \min\{\omega_{11},\omega_{tt}\} \text { if } x_1y_t\in E(G),
\end{equation}
and
\begin{equation}\label{PR2}
2\omega_{1t}\leqslant \min\{\omega_{1k}, \omega_{kt}\} \text{ if } x_1y_k, x_ky_t \in E(G) \text { with } 1<k<t.
\end{equation}

We consider the following  three cases:

    {\it Case $1$}: If there exists an edge $x_py_p$ such that $x_1y_p\notin E(G)$, then we can assume that $p$ is the smallest integer such that $x_1y_p\notin E(G)$. We then claim that $y_p$ is a leaf. Indeed, if $y_p$ is not a leaf, then $x_qy_p\in E(G)$ for some  $q < p$. By the minimality of $p$, we see that  $x_1y_q\in E(G)$, so $x_1y_p\in E(G)$ by  Property $(4*)$, a contradiction. So $y_p$ is a leaf of $G$. Then, by considering the weighted graph $(G\setminus x_p)_\omega$, we  obtain the desired inequalities.

    {\it Case $2$}: If there exists an edge $x_py_p$ such that $x_py_t\notin E(G)$, then we can assume that $p$ is the largest integer such that $x_py_t\notin E(G)$. We then claim that $x_p$ is a leaf. Indeed, if $x_p$ is not a leaf, then $x_py_q\in E(G)$ for some  $p<q<t$.  By the maximality of $p$, we have  $x_qy_t\in E(G)$, so  $x_py_t\in E(G)$ by  Property $(4*)$, a contradiction. So $x_p$ is a leaf of $G$. Considering the weighted graph $(G\setminus y_p)_\omega$, we get the desired inequalities.

    {\it Case $3$}: If $x_py_t, x_1y_p \in E(G)$ for all $p=2,\ldots, t-1$. Fix an integer $k$ such that $2\leqslant k \leqslant t-1$. By applying Lemma \ref{reduceEdges} repeatedly, we obtain that $I(H_\omega)^2$ is unmixed, where $H = G[\{x_1,y_1,x_k,y_k,x_t,y_t\}]$ is the induced subgraph of $G$ on the set $\{x_1,y_1,x_k,y_k,x_t,y_t\}$ (see Figure \ref{wbh}).
    \begin{figure}[ht]
\includegraphics[scale=0.5]{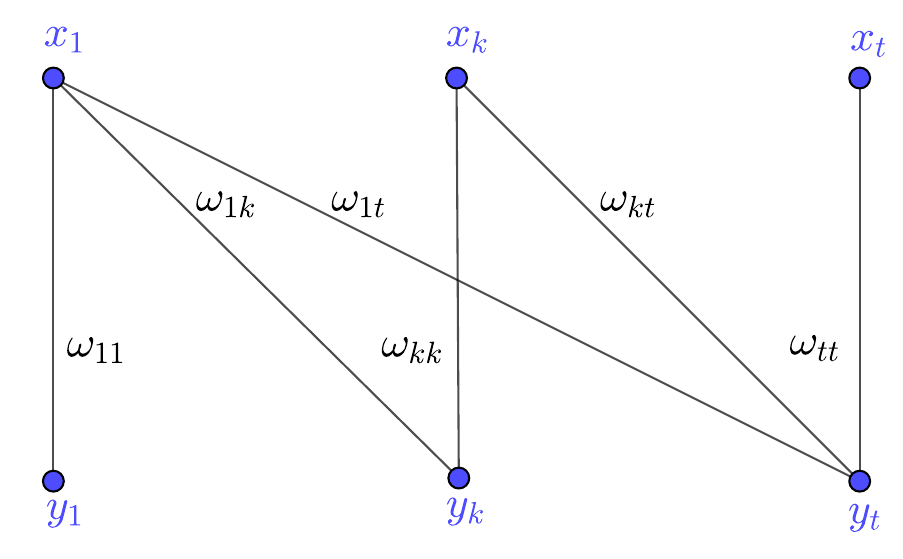}\\
\medskip
\caption{The weighted bipartite graph $H_\omega$.}
\label{wbh}
\end{figure}

Since $1<k<t$, by the induction hypothesis, we get
\begin{equation}\label{IFS}
    2\omega_{1k} \leqslant \min\{\omega_{11}, \omega_{kk}\} \text{ and } 2\omega_{kt} \leqslant \min\{\omega_{kk}, \omega_{tt}\}.
\end{equation}

Let $W = \{y_1,x_t\}$ and let $J = I(H_\omega)_W$ be the localization of $I(H_\omega)$ with respect to
$W$. Then $J$ is Cohen-Macaulay and 
\begin{equation}\label{FJ2}
J = (x_1^{\omega_{11}}, (x_1y_k)^{\omega_{1k}}, (x_1y_t)^{\omega_{1t}},(x_ky_k)^{\omega_{kk}}, (x_ky_t)^{\omega_{kt}}, y_t^{\omega_{tt}}).    
\end{equation}
So  $\sqrt{J} = (x_1,x_ky_k, y_t) = \p_1\cap \p_2$, where $\p_1=(x_1,x_k,y_t)$ and $\p_2=(x_1,y_k,y_t)$, we have $J^2 = J^{(2)}= Q_1^2\cap Q_2^2$, where $Q_1 = JR_{\p_1}\cap R$ and $Q_2=JR_{\p_2}\cap R$.

From Eq. (\ref{IFS}), we have $\omega_{1k} \leqslant \omega_{11}$, and so
\begin{align*}
JR_{\p_1} &= (x_1^{\omega_{11}}, x_1^{\omega_{1k}}, (x_1y_t)^{\omega_{1t}},x_k^{\omega_{kk}}, (x_ky_t)^{\omega_{kt}}, y_t^{\omega_{tt}})R_{\p_1}\\
&= (x_1^{\omega_{1k}}, (x_1y_t)^{\omega_{1t}},x_k^{\omega_{kk}}, (x_ky_t)^{\omega_{kt}}, y_t^{\omega_{tt}})R_{\p_1}.
\end{align*}
Hence $Q_1=JR_{\p_1} \cap R = (x_1^{\omega_{1k}}, (x_1y_t)^{\omega_{1t}},x_k^{\omega_{kk}}, (x_ky_t)^{\omega_{kt}}, y_t^{\omega_{tt}})$.

Similarly, 
\begin{align*}
JR_{\p_2} &= (x_1^{\omega_{11}}, (x_1y_k)^{\omega_{1k}}, (x_1y_t)^{\omega_{1t}},y_k^{\omega_{kk}}, y_t^{\omega_{kt}}, y_t^{\omega_{tt}})R_{\p_2}\\
&=(x_1^{\omega_{11}}, (x_1y_k)^{\omega_{1k}}, (x_1y_t)^{\omega_{1t}},y_k^{\omega_{kk}}, y_t^{\omega_{kt}})R_{\p_2},
\end{align*}
and so $Q_2 = JR_{\p_2} \cap R = (x_1^{\omega_{11}}, (x_1y_k)^{\omega_{1k}},(x_1y_t)^{\omega_{1t}},y_k^{\omega_{kk}}, y_t^{\omega_{kt}}).$

Let $f = x_1^{\omega_{11}}y_t^{\omega_{kt}}$. Since $\omega_{11} \geqslant 2\omega_{1k}$ by Eq. (\ref{IFS}), we have $f\in Q_1^2$. Note that $f\in Q_2^2$ since $x_1^{\omega_{11}},y_t^{\omega_{kt}}\in Q_2$, so $f\in J^2$.  Note that $y_t^{\omega_{tt}}\nmid f$ since $\omega_{tt}\ge 2\omega_{kt}>\omega_{kt}$ by Eq. (\ref{IFS}). Together with the fact that $f\in J^2$, this implies that either $x_1^{\omega_{11}} (x_1y_t)^{\omega_{1t}} \mid f$ or $(x_1y_t)^{2\omega_{1t}} \mid f$. The first case is impossible because $\deg_{x_1}(x_1^{\omega_{11}} (x_1y_t)^{\omega_{1t}}) > \deg_{x_1}(f)$. Hence  $(x_1y_t)^{2\omega_{1t}} \mid f$. It follows that $2\omega_{1t}\leqslant \omega_{11}$ and $2\omega_{1t}\leqslant \omega_{kt}$. 

Similarly, if we choose $g=x_1^{\omega_{1k}}y_t^{\omega_{tt}}$, then we can get $2\omega_{1t}\leqslant \omega_{tt}$ and $2\omega_{1t}\leqslant \omega_{1k}$. The proof  is  now complete.
\end{proof}

\begin{lem}\label{VC1} If $I(G_\omega)^2$ is Cohen-Macaulay, then
\begin{enumerate}
	 \item $2\omega(x_ix_j) \leqslant \min\{\omega(x_iy_i), \omega(x_jy_j)\}$ if $x_ix_j\in E(G)$  with $i\ne j$.
	\item $2\omega(x_ix_j)\leqslant \min\{\omega(x_iy_k), \omega(x_kx_j)\}$ for all different indices $i,j,k$ such that $x_iy_k$, $x_kx_j\in E(G)$. 
\end{enumerate}
\end{lem}

\begin{proof} $(1)$ We prove the ststement  by induction on $t\geqslant 2$. If $t = 2$, then it follows from Lemma \ref{twotimes}.
Suppose $t\geqslant 3$. Note that $i\ne j$. If $G$ has a pendant edge $x_py_p$ with $p\notin \{i,j\}$, then $I((G\setminus\{x_p,y_p\})_\omega)^2$ is Cohen-Macaulay by Lemma \ref{L2}. By applying the induction hypothesis on $H_\omega$, we obtain  the desired inequalities. Therefore, we  can assume that $x_py_p$ is not a pendant edge for $p\notin\{i,j\}$, and hence $x_iy_i$ and $x_jy_j$ are the only pendant edges by Lemma \ref{leaf2}. In particular, $y_i$ and $y_j$ are the only leaves of $G$.  
Therefore, by Lemma \ref{bp}, $G$ is bipartite. Thus, by \cite[Theorem $3.4$]{HH}, $G$ has a bipartition $(U,V)$ that satisfies the condition $(*)$. Furthermore, if $U =\{u_1,\ldots,u_t\}$ and $V = \{v_1,\ldots,v_t\}$,  by \cite[Theorem 0.2]{CRT},  $G$ has a unique perfect matching, we have
$$\{x_1y_1,\ldots,x_ty_t\} = \{u_1v_1,\ldots,u_tv_t\}.$$
Since $u_1v_1$ and $u_tv_t$ are pendent edges, it follows that $\{x_iy_i,x_jy_j\} = \{u_1v_1,u_tv_t\}$.  We can assume that $x_iy_i = u_1v_1$ as edges of $G$ so that $x_jy_j = u_tv_t$. Since $v_1, u_t$ are leaves, we deduce that $y_i = v_1$ and $y_j = u_t$. Then,  the inequality now becomes $2\omega(u_1v_t) \leqslant \min\{\omega(u_1v_1), \omega(u_tv_t)\}$,  which  follows from  Lemma \ref{TwoQ}.

$(2)$ The proof is almost the same as in the previous case. We prove by induction on $t\geqslant 3$. If $t = 3$, since $x_ky_k$ is not a pendent edge, we deduce that $G$ is a bipartite graph by Lemma \ref{bp}. Actually, we can check that $(\{x_i,x_k,y_j\}, \{y_i,y_k,x_j\})$ is a  bipartition of $G$ and it has a representation as in Figure \ref{WGH2}. With this configuration of $G$, the desired inequality follows from Lemma \ref{TwoQ}.

\begin{figure}[ht]
\includegraphics[scale=0.45]{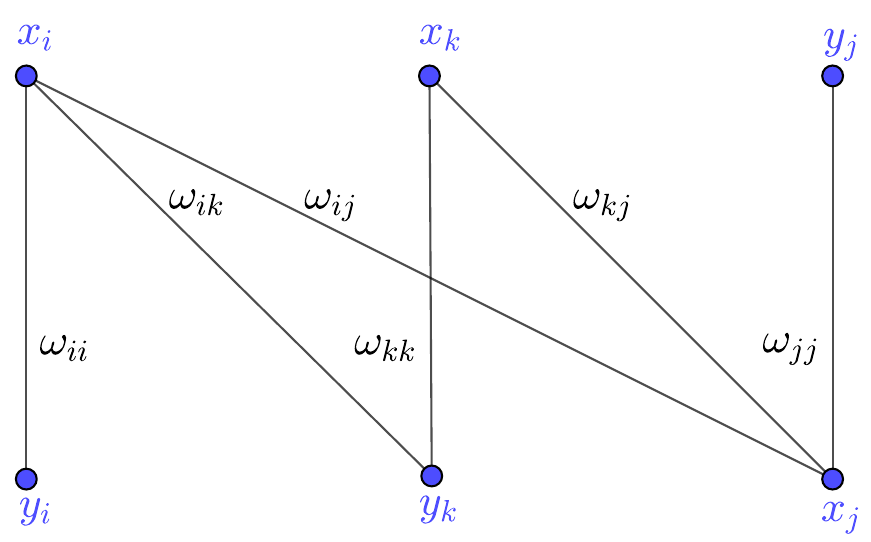}\\
\medskip
\caption{The graph $G_\omega$ after swapping the endpoints of $x_jy_j$.}
\label{WGH2}
\end{figure}

Suppose that $t\geqslant 4$. If $G$ has a pendant edge $x_py_p$ with $p\notin \{i,j\}$, then $I((G\setminus\{x_p,y_p\})_\omega)^2$ is Cohen-Macaulay by Lemma \ref{L2}. Applying the induction hypothesis to $H_\omega$, we obtain  the desired inequalities. Hence we can assume that $x_py_p$ is not a pendant edge for $p\notin\{i,j\}$, and hence $x_iy_i$ and $x_jy_j$ are the only pendant edges of $G$ by Lemma \ref{leaf2}. In particular, $y_i$ and $y_j$ are the only leaves of $G$.  In this case,  $G$ is bipartite by Lemma \ref{bp}. 

By the same argument as in the previous case, $G$ has a bipartition $(U,V)$ that satisfies the condition $(*)$, and if $U =\{u_1,\ldots,u_t\}$ and $V = \{v_1,\ldots,v_t\}$, then
$$\{x_1y_1,\ldots,x_ty_t\} = \{u_1v_1,\ldots,u_tv_t\}.$$

Since $u_1v_1$ and $u_tv_t$ are pendent edges of $G$, it follows that $\{x_iy_i,x_jy_j\} = \{u_1v_1,u_tv_t\}$. So either $x_iy_i = u_1v_1$ and $x_jy_j=u_tv_t$ or $x_iy_i = u_tv_t$ and $x_jy_j=u_1v_1$. This also implies that $x_ky_k = u_sv_s$ for some $s$ with $1 < s < t$. We now consider two possible cases:

{\it Case $1$}: $x_iy_i = u_1v_1$ and $x_jy_j=u_tv_t$. Then $y_i=v_1$ and $y_j = u_t$, since $v_1$ and $u_t$ are the only leaves of $G$, and then $x_i = u_1$ and $x_j = v_t$. Since $x_kx_j\in E(G)$, we conclude that $x_k = u_s$ and $y_k = v_s$. The desired inequality now becomes $2\omega(u_1v_t)\leqslant\min\{\omega(u_1v_s),\omega(u_sv_t)\}$. In this context, this inequality follows from Lemma \ref{TwoQ}.

{\it Case $2$}: $x_iy_i = u_tv_t$ and $x_jy_j=u_1v_1$. We follow the same proof as in the previous case, and the proof  is now complete.
\end{proof}

The following lemma can be proved by similar arguments as in the proof of  Lemma \ref{VC1}, we omit its proof.

\begin{lem}\label{VC2} If $I(G_\omega)^2$ is Cohen-Macaulay, then the following conclusions
hold:
\begin{enumerate}
\item $2\omega(x_iy_j) \leqslant \min\{\omega(x_iy_i), \omega(x_jy_j)\}$ if $x_iy_j\in E(G)$ with $i\ne j$.
\item $2\omega(x_iy_j)\leqslant \min\{\omega(x_iy_k), \omega(x_ky_j)\}$ for all different indices $i,j,k$ such that $x_iy_k$, $x_ky_j\in E(G)$.
\end{enumerate}
\end{lem}

\begin{lem}\label{bUM} Suppose that  $G_\omega$ satisfies all the conditions $(1)$ and $(2)$ in Lemmas \ref{VC1} and \ref{VC2}. Then $I(G_\omega)^2$ is unmixed.
\end{lem}
\begin{proof}
Let $I=I(G_\omega)$. Suppose that  $I^2$ is not unmixed, then there exists an associated prime ideal $\p$ of $I^2$ containing at least $t+1$ variables. In particular, it contains $x_k,y_k$ for some $k$. Let $f$ be a monomial such that $\p = I^2 \colon f$. Then
$fx_k = uf_1f_2$ and $fy_k = vg_1g_2$, where $f_1,f_2,g_1,g_2$ are monomial generators of $I$ and $u,v$ are monomials in $R$. Note that $f\notin I^2$, so $x_k \nmid u$ and $y_k \nmid v$. We next prove the following claim.

{\it Claim }: $(x_ky_k)^{\omega(x_ky_k)}\notin\{f_1,f_2,g_1,g_2\}$. Assume  by contradiction that the monomial $(x_ky_k)^{\omega(x_ky_k)}$ is in this set. If $(x_ky_k)^{\omega(x_ky_k)}\in\{f_1,f_2\}$, then we can assume that $f_1=(x_ky_k)^{\omega(x_ky_k)}$. It follows that $f=ux_k^{\omega(x_ky_k)-1}y_k^{\omega(x_ky_k)}f_2$ and 
\begin{equation}\label{F1}
 y_k f= ux_k^{\omega(x_ky_k)-1}y_k^{\omega(x_ky_k)+1}f_2= vg_1g_2.
\end{equation}
This gives $y_k^{\omega(x_ky_k)+1}\mid g_1g_2$ since $y_k \nmid v$. In particular, $\deg_{y_k}(g_1g_2)\geqslant \omega(x_ky_k)+1$. Using  the conditions $(1)$ in Lemmas $(\ref{VC1})$ and $(\ref{VC2})$, we conclude that either $g_1=(x_ky_k)^{\omega(x_ky_k)}$ or $g_2=(x_ky_k)^{\omega(x_ky_k)}$. We can assume that $g_1=(x_ky_k)^{\omega(x_ky_k)}$, then $g_1=f_1$.
Since $(x_ky_k)^{\omega(x_ky_k)} \mid fx_k$ and $(x_ky_k)^{\omega(x_ky_k)} \mid fy_k$, it follows that $y_k^{\omega(x_ky_k)} \mid f$ and $x_k^{\omega(x_ky_k)} \mid f$, respectively. So $f_1\mid f$. Next, we  let $h=f/f_1$. Then $hx_k = uf_2\in I$ and $hy_k=vg_2\in I$. 

On the other hand,  by the conditions  $(1)$ and $(2)$ in Lemmas $(\ref{VC1})$ and $(\ref{VC2})$ and \cite[Theorem 3.4] {SSTY}, we see that $I$ is unmixed. In particular,  every associated prime ideal of $I$ contains either $x_k$ or $y_k$, but not both, by Lemma \ref{cmCover}. It follows that $h\in I$. Hence, $f = hf_1\in I^2$, a contradiction.

Similarly, the case $(x_ky_k)^{\omega(x_ky_k)}\in\{g_1,g_2\}$ also leads to a contradiction. This proves the claim.

\medskip

Now, since $y_k\nmid v$ and $fy_k = vg_1g_2$,   $y_k \mid g_1g_2$, and so we can assume $y_k\mid g_1$. Together with the above  claim, this gives $g_1 = (x_iy_k)^{\omega(x_iy_k)}$ for some  $i < k$. Hence $x_i^{\omega(x_iy_k)}\mid f$ by the expression $fy_k = vg_1g_2$.
Similarly, using $fx_k = uf_1f_2$ and $x_k\nmid u$, we can suppose that $x_k \mid f_1$. Since $f_1\ne (x_ky_k)^{\omega(x_ky_k)}$ by the above  claim, we have either $f_1=(x_kx_j)^{\omega(x_kx_j)}$ for some $j\ne k$, or $f_1= (x_ky_j)^{\omega(x_ky_j)}$ for some $k< j$.

If $f_1= (x_kx_j)^{\omega(x_kx_j)}$ for some $j\ne k$, then  $x_j^{\omega(x_kx_j)}\mid f$. Since $g_1=(x_iy_k)^{\omega(x_iy_k)}$, $x_iy_k\in E(G)$,  we  have $x_ix_j\in E(G)$ and $i\ne j$ by Property $(5*)$. By the inequalities in Lemma \ref{VC1}, $2\omega(x_ix_j)\leqslant \min\{\omega(x_iy_k), \omega(x_kx_j)\}$. Thus $(x_ix_j)^{2\omega(x_ix_j)} \mid f$, and hence $f\in I^2$,  a contradiction.

If $f_1= (x_ky_j)^{\omega(x_ky_j)}$ for some $k< j$, then $x_ky_j\in E(G)$.  Note that $x_iy_k\in E(G)$, since $g_1 =(x_iy_k)^{\omega(x_iy_k)}$.  This implies that $x_iy_j\in E(G)$ by Property $(4*)$. By   Lemma  \ref{VC2}, we get $(x_iy_j)^{2\omega(x_iy_j)}\mid f$,  and so $f\in I^2$, a contradiction.

In summary, all cases are impossible. So  $I^2$ is unmixed, as required.
\end{proof}

Now we are ready to prove the  main result of this section.

\begin{thm}\label{squarePowers} Let $G$ be a very well-covered graph with $2t$ vertices and let $\omega$ be a weight function on $E(G)$. Further, assume that the vertices of $G$ are labeled in such a way
that the condition $(*)$ at the beginning of this section holds. Then $I(G_\omega)^2$ is Cohen-Macaulay if and only if the following conditions hold:
\begin{enumerate}
\item $2\omega(x_iz_j) \leqslant \min\{\omega(x_iy_i), \omega(x_jy_j)\}$ for each $x_iz_j\in E(G)$ with $i\ne j$ and $z_j\in\{x_j,y_j\}$.
\item $2\omega(x_iz_j)\leqslant \min\{\omega(x_iy_k), \omega(x_kz_j)\}$ for all distinct indices $i,j,k$ such that $x_iy_k$, $x_kz_j\in E(G)$ with $z_j\in\{x_j,y_j\}$.
\end{enumerate}
\end{thm}
\begin{proof}  $(\Longrightarrow)$: This follows from Lemmas \ref{VC1} and \ref{VC2}.

$(\Longleftarrow)$:  We prove the statement by induction on  $m(G_\omega): = \sum_{i=1}^t \deg_G(y_i)$. The base case where $m(G_\omega) = t$  follows from Theorem \ref{Pn}. Now suppose that $m(G_\omega)>t$. There exists some  $\ell$ such that $\deg_G(y_\ell)>1$. We can construct a  new very well-covered graph $G'$ from $G$ such that $V(G')=V(G)$ and the edge set is
$$E(G') = \{E(G) \setminus \{x_iy_\ell \mid x_iy_\ell\in E(G) \text{ with } i < \ell\}\} \cup \{x_ix_\ell\mid x_iy_\ell\in E(G) \text{ with } i < \ell\}.$$
We can check that $G'$ satisfies the condition $(*)$, so it is also very well-covered.

Next, we define the weight function  $\omega'$ on $E(G')$ by
$$
\omega'(e)=
\begin{cases}
\omega(x_iy_\ell) &\text{ if }  e =x_ix_\ell \text{ and } x_iy_\ell \in E(G) \text{ with } i < \ell,\\
\omega(e) &\text{ otherwise}.
\end{cases}
$$
We can check that the weight function  $\omega'$ on $E(G')$ satisfies all the  conditions of our theorem and that $m(G'_{\omega'})< m(G_\omega)$.  By the induction hypothesis, $I(G'_{\omega'})^2$ is Cohen-Macaulay. Since
$$(R/I(G_\omega)^2)/(y_\ell-x_\ell) \cong (R/I(G'_{\omega'})^2)/(y_\ell-x_\ell)$$
and $y_1-x_1,\ldots,y_t-x_t$ is a system of parameters for both $R/I(G_\omega)^2$ and $R/I(G'_{\omega'})^2$. This implies that $(R/I(G_\omega)^2)/(y_\ell-x_\ell)$ is Cohen-Macaulay.

So $I(G_\omega)^2$ is Cohen-Macaulay if $y_\ell-x_\ell$ is regular on $R/I(G_\omega)^2$. We now prove that $y_\ell-x_\ell$ is  indeed regular on $R/I(G_\omega)^2$. Since $I(G_\omega)^2$ is unmixed by Lemma \ref{bUM}, we have
$$\ass(R/I(G_\omega)^2) = \ass(R/I(G)).$$

On the other hand,  by Lemma \ref{cmCover}, $y_\ell-x_\ell$ does not belong to the set of  associated prime ideals of $I(G)$, so
$$y_\ell-x_\ell\notin \bigcup_{\p\in \ass(R/I(G_\omega)^2)}\p.$$
So  $y_\ell-x_\ell$ is regular on $R/I(G_\omega)^2$, hence $R/I(G_\omega)^2$ is Cohen-Macaulay. The proof of the theorem is now  complete.
\end{proof}

\section{Higher powers of edge ideals of weighted graphs}
\label{sec:Higher}

In this section, we study the Cohen-Macaulayness of higher powers of $I(G_\omega)$. Actually, Theorem \ref{Pn} deals with this problem for the case where $G$ has a perfect matching consisting of  pendant edges. Note that this graph is a very well-covered graph. The following result generalizes this theorem for any very well-covered graph.

\begin{thm}\label{Tk} Let $G$ be a very well-covered graph satisfying the condition $(*)$ at the beginning of Section \ref{sec:Second}. Suppose  there exists an integer $\ell\geqslant 1$ such that
 the following conditions hold:
 \begin{enumerate}
     \item $\ell\omega(x_iz_j) \leqslant \min\{\omega(x_iy_i), \omega(x_jy_j)\}$ for each $x_iz_j\in E(G)$ with $i\ne j$ and $z_j\in\{x_j,y_j\}$.
    \item $\ell\omega(x_iz_j)\leqslant \min\{\omega(x_iy_k), \omega(x_kz_j)\}$ for all distinct indices $i,j,k$ such that $x_iy_k$, $x_kz_j\in E(G)$ with $z_j\in\{x_j,y_j\}$.
\end{enumerate}
Then $I(G_\omega)^n$ is Cohen-Macaulay for all $n=1,\ldots,\ell$.
\end{thm}
\begin{proof}
 This can be shown by arguments similar to the only part of   Theorem \ref{squarePowers}, and we sketch the main ideas of the proof here.

The first step is to prove that $I(G_\omega)^n$ is  unmixed. To prove this,  we do induction on $n$. If $n = 1$, this is  done by \cite[Theorem 3.4] {SSTY}. If $n>1$, we do the same  as in the proof of Lemma \ref{bUM}.

The second step is to prove that $I(G_\omega)^n$ is Cohen-Macaulay by doing the same proof as in the proof of Theorem \ref{squarePowers}, using the fact that $I(G_\omega)^n$ is unmixed, which was proved in the previous step.
\end{proof}

By virtue of Theorem \ref{Tk}, for any  $\ell\geqslant 1$ and   any very well-covered graph $G$, we can construct a weight function  $\omega$ on $E(G)$ such that $I(G_\omega)^n$ is Cohen-Macaulay for all $n =1,\ldots, \ell$. However, a characterization of $G_\omega$ for which $I(G_\omega)^n$ is Cohen-Macaulay for a given integer $n$ is still unknown.

\medskip
Our calculation supports the following conjecture.

\begin{conj} Let $n\geqslant 2$ be an integer and let  $G$ be a very well-covered graph and $\omega$ be a weight function  on $E(G)$. Then  $I(G_\omega)^n$ is Cohen-Macaulay if and only if $I(G_\omega)^2$ is Cohen-Macaulay and $I(G_\omega)^n$ is unmixed.
\end{conj}

In the rest of this section, we are interested in characterizing $G_\omega$ such that  $I(G_\omega)^n$ is Cohen-Macaulay for all $n\geqslant 1$ in the case where  $G$ is a tree. The structure of the graph $G$ is given in \cite[Theorem 6.3.4]{Vi}: $G$ has a perfect matching $\{x_1y_1,\ldots, x_ty_t\}$, where each $y_i$ is a leaf vertex, and $G[x_1,\ldots,x_t]$ is a tree.

\medskip
For the case $t = 2$, $G[x_1,x_2]$ is  only one  edge, and we have the following result.

\begin{prop}\label{n2} Let $G_\omega$ be a weighted tree with the edge set $E(G) = \{ab,ax,by\}$. Then $I(G_\omega)^n$ is Cohen-Macaulay for all $n\geqslant 1$ if and only if $\min\{\omega(ax), \omega(by)\} \geqslant 2 \omega(ab)$.
\end{prop}
\begin{proof} For simplicity, let $I = I(G_\omega)$ and $p=\omega(ax), k= \omega(ab), q = \omega(by)$. Then  the weighted graph $G_\omega$ is shown as in Figure \ref{pqTree} and $I = (a^px^p,a^kb^k, b^qy^q)$.

If $I^n$ is Cohen-Macaulay for all $n\geqslant 1$, then $\min\{p,q\} \geqslant 2k$ by Lemma \ref{twotimes}.

Conversely, suppose that $\min\{p,q\}\geqslant 2k$. Then $I$ has an irredundant primary decomposition
$$I = (a^kb^k, a^p,b^q) \cap (a^k,y^q) \cap (b^k,x^p).$$

For any integer $n\geqslant 1$, we will prove that $I^n$ is Cohen-Macaulay  in the following two steps.

{\it Step $1$}: We prove that $I^n = I^{(n)}$, where $I^{(n)}$ is the $n$-th symbolic power of $I$. Let $f$ be any monomial of $I^{(n)}$. Since $I^{(n)}=(a^kb^k, a^p,b^q)^n \cap (a^k,y^q)^n \cap (b^k,x^p)^n$, we can write $f = ga^{ku}b^{kv}x^{(n-v)p}y^{(n-u)q}$, where $g$ is a monomial in $R$ and $u,v \in \{0,\ldots,n\}$ such that $ga^{ku}b^{kv}\in (a^kb^k, a^p,b^q)^n$. So we can write
$$ga^{ku}b^{kv} = h(a^kb^k)^s (a^p)^i (b^q)^j,$$
where $h$ is a monomial in $R$ and $s+i+j = n$. We choose such an expression  for which $i+j$ is minimal. If $i\geqslant 1$ and $j \geqslant 1$, then
$$ga^{ku}b^{kv} = (ha^{p-2k}b^{q-2k})(a^kb^k)^{s+2} (a^p)^{i-1} (b^q)^{j-1}$$
since $\min\{p,q\}\geqslant 2k$. This contradicts the minimality of $i+j$, so either $i = 0$ or $j=0$. Without loss of generality, we can assume that $j = 0$.

If $i = 0$, then $s = n$ and  $ga^{ku}b^{kv} = h(a^kb^k)^n$. It follows that $f=ga^{ku}b^{kv}x^{(n-v)p}y^{(n-u)q}$ $= h(a^kb^k)^nx^{(n-v)p}y^{(n-u)q}\in I^n$. Suppose  $i > 0$. Then we have $b^k \nmid h$. Because if $b^k \mid h$,  we can write
$$ga^{ku}b^{kv} = h'(a^kb^k)^{s+1} (a^p)^{i-1}$$
where $h' = a^{p-k}h/b^k$, which also contradicts the minimality of $i+j$. So $b^k \nmid h$. Together with the equality $ga^{ku}b^{kv} = h(a^kb^k)^s (a^p)^i$, this forces $v\leqslant s$, so
$i = n-s \leqslant n-v$. It follows that
$$f = h(a^kb^k)^s (a^p)^i x^{(n-v)p}y^{(n-u)q}=hy^{(n-u)q}x^{(n-v-i)p} (a^kb^k)^s (a^px^p)^i\in I^n,$$
since $i+s=n$. So we always have $f\in I^n$, and hence $I^n = I^{(n)}$.

{\it Step $2$}:  We prove that $I^n$ is Cohen-Macaulay. Suppose, on the contrary, that $I^n$ is not Cohen-Macaulay.  Since $I^n = I^{(n)}$, it is unmixed, so $\sqrt{I^n\colon f}$ is not Cohen-Macaulay for some monomial $f\notin I^n$ by Lemma \ref{lemCM}.

Note that $\ass(R/I^n) = \{(a,b),(a,y),(b,x)\}$. So $\sqrt{I^n\colon f}$ is not Cohen-Macaulay if and only if $\sqrt{I^n\colon f} = (a,y) \cap (b,x)$. This means that
$$f\in (a^kb^k, a^p,b^q)^n, f \notin  (a^k,y^q)^n, \text{ and } f \notin (b^k,x^p)^n.$$

Since $f\in (a^kb^k, a^p,b^q)^n$, we can write $f = f_1(a^kb^k)^s(a^p)^i(b^q)^j$, where $f_1$ is a monomial in $R$ and $s+i+j = n$. Choose such an expression for which $i+j$ is minimal. By the same argument as in the previous step, we have either $i=0$ or $j=0$. Without loss of generality, we can assume that $j =0$. It follows that $f =f_1(a^kb^k)^s(a^p)^i =f_1b^{ks}a^{(p-k)i}(a^k)^n\in (a^k,y^q)^n$, a contradiction. So $I^n$ is Cohen-Macaulay, and the proof is now complete.
\end{proof}

We now consider the case $t=3$, in which case $G[x_1,x_2,x_3]$ is just a path of length $3$. For simplicity, consider the weighted tree $G_\omega$ with the edge set $$E(G) = \{ab,bc,ax,by,cz\}$$
and the weight  function
$$k = \omega(ab), m = \omega(bc), p = \omega(ax), q = \omega(by), r = \omega(cz),$$
which is shown in Figure \ref{pqrTree}.

\begin{figure}[ht]
\includegraphics[scale=0.5]{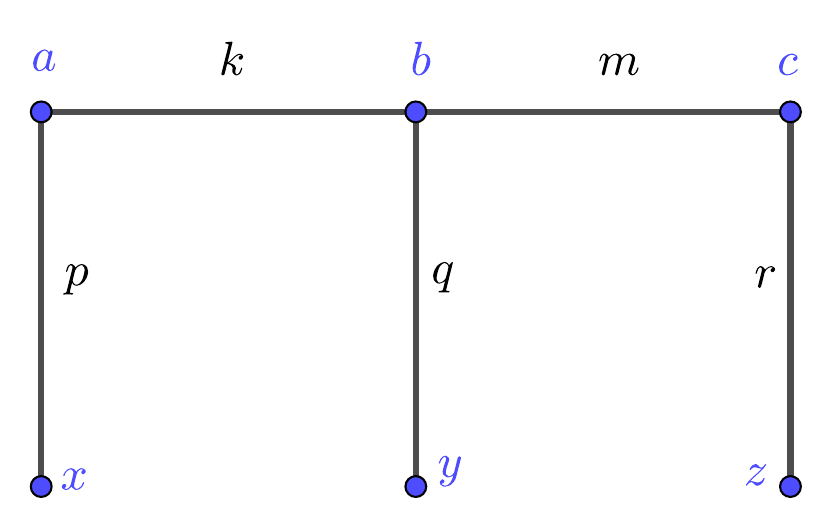}\\
\medskip
\caption{The weighted graph $G_\omega$.}
\label{pqrTree}
\end{figure}   

\begin{lem}\label{dif} Let $G_\omega$ be a weighted graph as shown in Figure \ref{pqrTree}. If $m = k$, then $I(G_\omega)^n$ is not Cohen-Macaulay for all $n\geqslant \max\{p,r\}/m+2$.
\end{lem}
\begin{proof} Let $I = I(G_\omega)$, then $I = ((ab)^m, (bc)^m, (ax)^p,(by)^q,(cz)^r)$. Suppose by contradiction that $I(G_\omega)^n$ is Cohen-Macaulay for some  $n\geqslant \max\{p,r\}/m+2$.

First, we claim  that $\min\{p,q,r\} \geqslant 2m$. Indeed, let $H = G\setminus c$ so that $I(H_\omega) = ((ab)^m, (ax)^p, (by)^q)$. By Lemma \ref{L2}, $I(H_\omega)^n$ is Cohen-Macaulay, and hence $\min\{p, q\} \geqslant 2m$ by Lemma \ref{twotimes}. Similarly, by considering the graph $G\setminus a$, we obtain $r \geqslant 2m$, and the claim follows.

Let $J = I_{\{y\}}$ be  the localization of $I$ with respect to the set $\{y\}$, then $J= ((ab)^m, (bc)^m,(ax)^p,
b^q,(cz)^r)$. Using the  above claim, we can check that 
$$J = ((ab)^m, (bc)^m, a^p,b^q,c^r) \cap (a^p,b^m,z^r)\cap(b^m,c^r,x^p)\cap (b^m,x^p,z^r)$$
is an irredundant decomposition.

Since $I^n$ is Cohen-Macaulay,  $J^{n}$ is also Cohen-Macaulay by Lemma \ref{localization}. Consequently,
$$J^n = ((ab)^m, (bc)^m, a^p,b^q,c^r)^n \cap (a^p,b^m,z^r)^n\cap(b^m,c^r,x^p)^n\cap (b^m,x^p,z^r)^n.$$

Let $f=(b^m)^{n-1}a^p c^r$, then $f\in (a^p,b^m,z^r)^n\cap(b^m,c^r,x^p)^n$ and
$f \notin (b^m,x^p,z^r)^n$. Now we show that $f\notin ((ab)^m, (bc)^m, a^p,b^q,c^r)^n$.  If this is not the case, then $f$ could be written as
$$f = g((ab)^m)^s((bc)^m)^t(a^p)^u(b^q)^v(c^r)^w,$$
where $s,t,u,v,w\in \N$ with $s+t+u+v+w = n$ and $g$ is a monomial in $R$. Therefore,
$$
\begin{cases}
    p \geqslant ms + pu\\
    m(n-1)\geqslant ms+mt+qv\\
    r\geqslant mt + rw
\end{cases}
$$
The inequality $p \geqslant ms + pu$ forces $u\leqslant 1$. Similarly,  the inequality $r\geqslant mt + rw$ forces  $w\leqslant 1$.  So we have four possible cases:

{\it Case $1$}: If $u = w = 0$, then $s+t+v = n$. Since $q \geqslant 2m$, we have
$$m(n-1)\geqslant ms+mt+qv \geqslant ms+mt+2mv = m(s+t+v)+mv \geqslant mn,$$
a contradiction.

{\it Case $2$}: If $u=1$ and  $w = 0$, then $s+t+v = n-1$, and then
$$m(n-1) \geqslant ms+mt+2mv = m(s+t+v)+mv \geqslant m(n-1)+mv.$$
This implies $v = 0$. Since $u=1$,  combined with the inequality $p\geqslant ms +pu$, we get $s=0$. Thus
$t = n-(s+u+v+w) = n-1$.

On the other hand, since $r\geqslant mt+rw$ and $w=0$, we get $n = t+1 \leqslant r/m + 1$. This contradicts the assumption that  $n\geqslant \max\{p,r\}/m+2$.

{\it Case $3$}: When $u=0$ and $w = 1$. We derive a contradiction by the same argument as in the previous case.

{\it Case $4$}: If $u=w= 1$, then, in this case, we must have $s = t =0$. Hence, $v = n-(s+t+u+w)=n-2$, and hence
$m(n-1)\ge ms+mt+qv \geqslant q(n-2)\geqslant 2m(n-2)$. This forces $n\leqslant 3$, a contradiction.

\medskip

To sum up, all cases are impossible. It follows that $f\notin ((ab)^m, (bc)^m, a^p,b^q,c^r)^n$. Hence $J^n \colon f = ((((ab)^m, (bc)^m, a^p,b^q,c^r)^n)\colon f) \cap  ((b^m,x^p,z^r)^n\colon f)$,
and hence $\sqrt{J^n\colon f} = (a,b,c) \cap (b,x,z)$. In particular, $\sqrt{J^n\colon f}$ is not Cohen-Macaulay, it contradicts Lemma \ref{lemCM}.  Therefore, $I(G_\omega)^n$ is not Cohen-Macaulay for all $n\geqslant \max\{p,r\}/m + 2$, and the proof is complete.
\end{proof}

\begin{lem}\label{pqr} Let $G_\omega$ be a weighted tree as  shown in Figure \ref{pqrTree}, with $k < m$. If $I(G_\omega)^n$ is Cohen-Macaulay for all $n\geqslant 1$, then $p\geqslant k\lceil m/(m-k)\rceil$, where  $\lceil m/(m-k)\rceil$ is the smallest integer $\ge  m/(m-k)$.
\end{lem}
\begin{proof}
For simplicity, let $I = I(G_\omega)$, then $I = (a^kb^k, b^mc^m, a^px^p,b^qy^q,c^rz^r)$.
 Since $I^n$ is Cohen-Macaulay for all $n\geqslant 1$, we have $2k \leqslant \min\{p,q\}$ and $2m \leqslant \min\{q,r\}$ by Theorem \ref{squarePowers}.

Let $J = I_{\{c,y\}}$ be  the localization of $I$ with respect to the set 
$\{c,y\}$. Since $k < m < \min\{q,r\}$, we have
$$J = (a^kb^k, b^m, a^px^p,z^r).$$
Note also that $J$ has an irredundant primary decomposition as follows
$$J = (a^kb^k,a^p,b^m,z^r) \cap (b^k,x^p,z^r).$$

Let $f = a^pb^{m(m-1)}$, then $f \in (a^kb^k,a^p,b^m,z^r)^m$. On the other hand, since $k < m$, we have $m(m-1)\geqslant mk$, so $f \in (b^k,x^p,z^r)^m$. It follows that $f\in J^{(m)}$. In particular, $f \in J^m$ since $J^m$ is unmixed. So we can write $f$ as
$$f = g(a^kb^k)^s (b^m)^t$$
with $s+t = m$ and $g$ is a monomial in $R$. This yields
$$
\begin{cases}
    p \geqslant ks,\\
    m(m-1)\geqslant ks+mt.
\end{cases}
$$

Since $s+t = m$, the second inequality of the above system  gives $m(m-1)\geqslant ks+m(m-s)$, or equivalently $(m-k)s \geqslant m$. Hence, $s\geqslant m/(m-k)$, and hence $s\geqslant \lceil m/(m-k)\rceil$ since $s$ is an integer. Hence $p\geqslant ks \geqslant k\lceil m/(m-k)\rceil$, as required.
\end{proof}

\begin{lem}\label{rBound} Let $G_\omega$ be a weighted tree as shown in Figure \ref{pqrTree}, with $k<m$. If $I(G_\omega)^n$ is Cohen-Macaulay for all $n\geqslant 1$, then $r \geqslant m(\lceil m/(m-k)\rceil-2)$.
\end{lem}
\begin{proof} Let $I = I(G_\omega)$ and $i=\lceil m/(m-k)\rceil-1=\lceil k/(m-k)\rceil$. If $i=1$, the lemma is trivial, so we assume that $i\geqslant 2$. By Theorem \ref{squarePowers}, $\min\{p,q\}\geqslant 2k$ and $\min\{q,r\}\geqslant 2m$. By Lemma \ref{pqr}, we have
$$p\geqslant k\left\lceil \frac{m}{m-k}\right\rceil\geqslant k\frac{m}{m-k} = m\frac{k}{m-k}\geqslant m\left(\left\lceil \frac{k}{m-k}\right\rceil-1\right) = m(i-1).$$

Let $J = I_{\{x,y\}}$ be  the localization of $I$ with respect to the set 
$\{x,y\}$,  then $J= (a^kb^k,b^mc^m, a^p,b^q,c^rz^r)$ and it has an irredundant primary decomposition as follows
$$J = (a^kb^k,b^mc^m, a^p,b^q,c^r) \cap (a^kb^k, a^p,b^m,z^r).$$

Let $f =  a^pb^{m(i-1)}c^r$. Since $p \geqslant m(i-1)$, we have $(a^mb^m)^{i-1}c^r \mid f$, and so $f\in (a^kb^k,b^mc^m, a^p,b^q,c^r)^i$. Obviously, $f \in (a^kb^k, a^p,b^m,z^r)^i$, so $f \in J^{(i)}$. In particular, $f \in J^i$ because $J^i$ is unmixed. Therefore, we can write $f$ as
$$f = g(a^kb^k)^u (b^mc^m)^v (a^p)^s (b^q)^t,$$
where $u+v+s+t = i$ and $g$ is a monomial in $R$. It follows that
$$
\begin{cases}
    p \geqslant ku+ps,\\
    m(i-1)\geqslant ku+mv+qt,\\
    r\geqslant mv.
\end{cases}
$$

The first inequality implies that $s = 0$ or $s=1$. If $s = 0$, then $u+v+t = i$. Since $q\geqslant 2m> k$, from the second inequality of the above system,  we have
$$m(i-1) \geqslant ku+mv+qt \geqslant ku+kv+kt = k(u+v+t) = ki,$$
and so $i\geqslant m/(m-k)$. On the other hand, since $i=\lceil k/(m-k)\rceil$, we have
$i <k/(m-k)+1 = m/(m-k)$, a contradiction. So $s=1$.

Since $s=1$, we get $u = 0$, and  hence $v+t = i-1$. The second inequality of the above system yields $m(i-1-v)\geqslant qt$. Note that $i-1-v = t$, so $mt\geqslant qt$. Since $m < q$, we get $t = 0$. So $v = i-1$. Together with the last inequality of the  above system, this forces $r\geqslant m(i-1)$, and the proof is now  complete.
\end{proof}

Now we  give  a necessary condition for a weighted tree $G_\omega$ such that $I(G_\omega)^n$ is Cohen-Macaulay for all $n\geqslant 1$.

\begin{prop} \label{CMn} Let $G_\omega$ be a weighted tree.  If $I(G_\omega)^n$ is Cohen-Macaulay for all $n\geqslant 1$, then $G$ has a perfect matching $\{x_1y_1$, $\ldots$, $x_ty_t\}$, where each $y_i$ is a leaf vertex in $G$. Let $\omega_{ij}=\omega(x_ix_j)$ if $x_ix_j\in E(G)$ and $m_i = \omega(x_iy_i)$ for all $i=1,\ldots,t$. Then the following conditions hold:
\begin{enumerate}
    \item $2\omega_{ij}\leqslant \min\{m_i, m_j\}$ for any $x_ix_j\in E(G)$,
    \item For any two different $x_ix_j,x_jx_k\in E(G)$, we have  $\omega_{ij}\ne \omega_{jk}$. Furthermore, if $\omega_{ij} < \omega_{jk}$, then
\begin{enumerate}
\item $m_i \geqslant \omega_{ij}\lceil \omega_{jk}/(\omega_{jk}-\omega_{ij})\rceil$, and
\item $m_k \geqslant \omega_{jk}(\lceil \omega_{jk}/(\omega_{jk}-\omega_{ij})\rceil-2)$.
\end{enumerate}    
\end{enumerate}
\end{prop}
\begin{proof} $(1)$ Since $I(G_\omega)^n$ is Cohen-Macaulay,  $I(G)=\sqrt{I(G_\omega)^n}$  is also Cohen-Macaulay by \cite[Theorem 2.6]{HTT}.
By \cite[Proposition 3.3]{CRT}, we get that $G$ has a perfect matching $\{x_1y_1$, $\ldots$, $x_ty_t\}$, where each $y_i$ is a leaf vertex in $G$. On the other hand, since $I(G_\omega)^2$ is Cohen-Macaulay, by Lemma \ref{VC1}, $\omega(x_iy_i) \geqslant 2\omega(x_iz)$ for all $i=1,\ldots, t$ and $z\in N_G(x_i)\setminus \{y_i\}$. This proves $(1)$.

$(2)$ Let $H = G \setminus \{x_s\mid s\notin\{i,j,k\}\}$, then $E(H)=\{x_iy_i,x_jy_j,x_ky_k,  x_ix_j,x_jx_k\}$. Since $I(G_\omega)^n$ is Cohen-Macaulay for all $n\geqslant 1$,  $I(H_\omega)^n$ is also Cohen-Macaulay for all $n\geqslant 1$ by Lemma \ref{L2}. Thus $\omega(x_ix_j)\ne \omega(x_jx_k)$ by Lemma \ref{dif}.
Therefore, $(a)$ and $(b)$ follow from Lemmas \ref{pqr} and \ref{rBound}, respectively.
\end{proof}

To characterize the weighted trees $G_\omega$ such that $I(G_\omega)^n$ are Cohen-Macaulay for all $n\geqslant 1$, we propose the following conjecture.

\begin{conj}\label{conjTrees} The converse of Proposition \ref{CMn} is true.
\end{conj}

Note that the conjecture is true for $t=2$ by Proposition  \ref{n2}. We will prove the conjecture for the case where $G[x_1,\ldots,x_t]$ is a star. Recall that a graph $T$ is called a star if it has a vertex $v$ such that $E(T) = \{vu \mid u\in V(T)\setminus \{v\}\}$.  In this case, the vertex $v$ is called the center of $T$. We begin by proving the unmixedness of  the edge ideals of such weighted graphs.

\begin{lem}\label{unmixed-star-center} Suppose $G[x_1,\ldots,x_t]$ is a star with the center $x_t$. Let $m_i =\omega(x_iy_i)$ for $i=1,\ldots,t$ and let $d_i = \omega(x_ix_t)$ for $i\ne t$. Suppose that 
\begin{enumerate}
    \item $m_t \geqslant 2\max\{d_i \mid i\ne t\}$.
    \item $d_i \ne d_k$ for any two different $x_ix_t,x_kx_t\in E(G)$. Furthermore, if $d_i < d_k$, then
\begin{enumerate}
\item $m_i \geqslant d_i\lceil d_k/(d_k-d_i)\rceil$, and
\item $m_k \geqslant d_k\max\{2,\lceil d_k/(d_k-d_i)\rceil-2\}$.
\end{enumerate}
\end{enumerate}
Then, $I(G_\omega)^n$ is unmixed for all $n\geqslant 1$.
\end{lem}
\begin{proof} Let $I = I(G_\omega)$, then
$$I =((x_ix_t)^{d_i}, (x_ky_k)^{m_k}\mid i\in [t] \setminus\{t\} \text{ and } k\in [t]).$$

We now prove by induction on $t$ that $I^n$ is unmixed. If $t = 2$, the lemma follows from Proposition \ref{n2}. Suppose $t\geqslant 3$.
We continue to prove the statement by induction on $n$. The case $n\leqslant 2$ follows from  Theorem \ref{Pn}. Suppose that $n\ge 3$ and $I^{n-1}$ is unmixed. Suppose by contradiction that $I^n$ is not unmixed, so  there exists an associated prime ideal $\p$ of $I^n$ containing at least $t+1$ variables. In particular, it contains  $x_i,y_i$ for some $i\in [t]$. Let $f$ be a monomial such that $\p = I^n \colon f$, then  $f\notin I^n$.
 Since $fx_i,fy_i \in I^n$, we can write $f x_i = gf_1\cdots f_n \text{ and } f y_i= h g_1\cdots g_n$, where $g,h$ are monomials and $f_i,g_i$ are monomial generators of $I$ for $i=1,\ldots,n$. We first prove the  following six claims:

\medskip

{\it Claim $1$}: $\{f_1,\ldots,f_n\} \cap \{g_1,\ldots,g_n\}=\emptyset$. Indeed, assuming by contradiction that this is not true, we can assume that $f_1=g_1$. First we show that $f_1\mid f$. It is obvious that $f_1\mid f$ if either $x_i\nmid f_1$ or $y_i\nmid f_1$, so we can assume  that $x_i\mid f_1$ and  $y_i\mid f_1$, it follows that $f_1= (x_iy_i)^{m_i}$. From $f_1 \mid fx_i$ we get $y_i^{m_i}\mid f$. Similarly, $x_i^{m_i} \mid f$, and hence $f_1\mid f$, as desired. Now, let $f' = f/f_1$. Then  $f'x_i = gf_2\cdots f_n \in I^{n-1}$ and $f'y_i = hg_2\cdots g_n \in I^{n-1}$.  
We can prove that $f^{\prime}\in I^{n-1}$. In fact, suppose that $I=Q_1\cap Q_2\cap\cdots\cap Q_s$ is an irredundant primary decomposition of $I$. 
By the induction hypothesis, $I^{n-1}$ is unmixed, so  $I^{n-1}=I^{(n-1)}$, where $I^{(n-1)}$ is the $(n-1)$-th symbolic power of $I$. In other words,  $I^{n-1}=Q_1^{n-1}\cap Q_2^{n-1}\cap\cdots\cap Q_s^{n-1}$. Conversely, if $f^{\prime}\notin I^{n-1}$, there exists some $ j\in [s]$ such that $f^{\prime}\notin Q_j^{n-1}$, i.e.  $\sqrt{Q_j^{n-1}:f^{\prime}}=\sqrt{Q_j}$. Thus forces
	\begin{eqnarray*}
		(I^{n-1}:f^{\prime})&=&(Q_1^{n-1}\cap Q_2^{n-1}\cap\cdots\cap Q_s^{n-1}):f^{\prime}\\
		&=&(Q_1^{n-1}:f^{\prime})\cap (Q_2^{n-1}:f^{\prime})\cap\cdots\cap(Q_s^{n-1}:f^{\prime})\\
		&\subseteq&(Q_j^{n-1}:f^{\prime})\subseteq\sqrt{Q_j^{n-1}:f^{\prime}}=\sqrt{Q_j}.\\
	\end{eqnarray*}
   Note that $\sqrt{Q_j}$ is an associated prime ideal of $I(G)$,  by Lemma \ref{cmCover}, we get $(x_i,y_i)\nsubseteq \sqrt{Q_j}$, which implies $(x_i,y_i)\nsubseteq (I^{n-1}:f^{\prime})$, i.e., $f^{\prime}x_i\notin I^{n-1}$ or $f^{\prime}y_i\notin I^{n-1}$, a contradiction. So $f^{\prime}\in I^{n-1}$.
  
\medskip

{\it Claim $2$}: $x_i \nmid g$ and $y_i\nmid h$. Indeed, if $x_i\mid g$, then $f=(g/x_i)f_1\cdots f_n\in I^n$. Similarly, if $y_i\mid h$, then $f\in I^n$. These two cases are impossible, since $f\notin I^n$, and the claim follows.

\medskip

{\it Claim $3$}: There is no index $k\ne t$ such that either $(x_kx_t)^{d_k}, (x_ky_k)^{m_k}\in \{f_1,\ldots,f_n\}$, or $(x_kx_t)^{d_k}, (x_ky_k)^{m_k}\in \{g_1,\ldots,g_n\}$. In fact, if this assertion is not true, we can assume that  $(x_kx_t)^{d_k}, (x_ky_k)^{m_k}\in \{f_1,\ldots,f_n\}$  for some $k\ne t$. By Claim $1$, we get $(x_kx_t)^{d_k},(x_ky_k)^{m_k}\notin \{g_1,\ldots,g_n\}$. It follows that $\gcd((x_ky_k)^{m_k},g_q)=1$ for $q=1,\ldots,n$. Hence  $(x_ky_k)^{m_k}\mid h$, since $ f y_i= h g_1\cdots g_n$. Let  $h= \lambda(x_ky_k)^{m_k}$ for some monomial $\lambda$, then $fy_i =\lambda (x_ky_k)^{m_k}g_1\cdots g_n $. This implies that $f\in I^n$, a contradiction, and the claim follows.

\medskip

{\it Claim $4$}: There exists some $p\in [n]$ such that  $g_p = (x_iy_i)^{m_i}$. Indeed, since $fy_i = hg_1\cdots g_n$ and $y_i\nmid h$ by  Claim $2$, we have $y_i\mid g_p$ for some $p\in [n]$, and hence $g_p = (x_iy_i)^{m_i}$.

\medskip

{\it Claim $5$}:  If $i\ne t$, then there exists some $p\in [n]$ such that $f_p = (x_ix_t)^{d_i}$.  Indeed, since $fx_i = gf_1\cdots f_n$ and $x_i\nmid g$ by  Claim $2$, we have $x_i\mid f_p$ for some $p\in [n]$. Note that $x_i\mid f_p$ if and only if  $f_p \in \{(x_iy_i)^{m_i}, (x_ix_t)^{d_i}\}$. By Claims $1$ and $4$, we get $f_p = (x_ix_t)^{d_i}$.

\medskip

{\it Claim $6$}: $\{(x_kx_t)^{d_k}\mid k\in [t] \setminus \{t\}\} \subseteq \{f_1,\ldots,f_n,g_1,\ldots,g_n\}$. 

In fact, 
if there exists some $k\in [t] \setminus \{t\}$ such that $(x_kx_t)^{d_k} \notin \{f_1,\ldots,f_n,g_1,\ldots,g_n\}$,
then  we can get  $k\ne i$. Indeed, the case $i = t$  is trivial. If $i\ne t$, then   $(x_ix_t)^{d_i}\in\{f_1,\ldots,f_n\}$ by Claim $5$, so  $k\ne i$.
Now we prove that  for such $k$, $ (x_ky_k)^{m_k}\notin \{f_1,\ldots,f_n,g_1,\ldots,g_n\}$. Conversely,  $(x_ky_k)^{m_k}\in \{f_1,\ldots,f_n\}$ or $(x_ky_k)^{m_k}\in\{g_1,\ldots,g_n\}$. We distinguish between the following two cases:

(i)  If  $(x_ky_k)^{m_k}\in \{f_1,\ldots,f_n\}$, then $(x_ky_k)^{m_k}\mid fx_i$ by the  expression of $fx_i$.  It follows that $(x_ky_k)^{m_k}\mid f$, since  $k\ne i$.  In particular, $(x_ky_k)^{m_k}\notin \{g_1,\ldots,g_n\}$ by Claim $1$.  Since $k\ne t$,   by Claim $5$, there exists some $p\in [n]$ such that  $f_p = (x_kx_t)^{d_k}$. Again, by Claim $1$, we get $(x_kx_t)^{d_k}\notin \{g_1,\ldots,g_n\}$. So   $x_k\nmid g_j$ and $y_k \nmid g_j$ for all  $j\in [n]$, since $N_G(x_k)=\{y_k,x_t\}$.
Since $(x_ky_k)^{m_k}\mid f$, we get $(x_ky_k)^{m_k}\mid fy_i=hg_1\cdots g_n$. This forces $(x_ky_k)^{m_k}\mid h$.
Let $h=\lambda(x_ky_k)^{m_k}$ for some monomial $\lambda$, then $fy_i=hg_1\cdots g_n=\lambda(x_ky_k)^{m_k}g_1\cdots g_n\in I^{n+1}$. This implies that  $f\in I^n$, a contradiction.

(ii)  If  $(x_ky_k)^{m_k}\in \{g_1,\ldots,g_n\}$, then, by similar arguments as in case (i) we also get  $f\in I^n$, a contradiction.

Therefore,  $(x_kx_t)^{d_k}, (x_ky_k)^{m_k}\notin \{f_1,\ldots,f_n,g_1,\ldots,g_n\}$.
In this case, let $H_\omega = G_\omega\setminus \{x_k,y_k\}$, then  $fx_i,fy_i\in I(H_\omega)^n$. Note that  $H_\omega$ is an induced subgraph of $G_\omega$ on the set $V(G_\omega)\setminus \{x_k,y_k\}$ and $G[x_1,\ldots,\hat{x}_k,\ldots, x_t]$ is a star with center $x_t$, where $\hat{x}_k$ denotes $x_k$ removed from $\{x_1,\ldots,x_t\}$. By the induction hypothesis, $I(H_\omega)^n$ is unmixed, so $f\in I(H_\omega)^n$ by the same argument as in Claim 1.  In particular,  $f\in I^n$, a contradiction. Hence  $\{(x_kx_t)^{d_k}\mid k\in [t] \setminus \{t\}\} \subseteq \{f_1,\ldots,f_n,g_1,\ldots,g_n\}$.

\medskip
Next, we will prove that $f\in I^n$ in the expression $\p = I^n \colon f$  of the associated prime ideal $\p$, which contradicts $f\notin I^n$. This implies that $I^n$ is unmixed. We distinguish into the following  two cases:

\medskip

{\it Case $1$}: If $i \ne t$, then $(x_ty_t)^{m_t} \notin \{g_1,\ldots,g_n\}$. Indeed, if this is not the case, then  we can assume that $g_{n-1}= (x_ty_t)^{m_t}$ and $g_n = (x_iy_i)^{m_i}$ by Claim $4$. Since $\min\{m_i,m_t\} \geqslant 2d_i$, we get $(x_ix_t)^{2d_i} \mid (x_iy_i)^{m_i} (x_ty_t)^{m_t}$, and hence $(x_iy_i)^{m_i} (x_ty_t)^{m_t}=y_i\lambda (x_ix_t)^{2d_i}$ for some monomial $\lambda$. It follows that
$$fy_i = y_i \lambda h g_1\cdots g_{n-2}  (x_ix_t)^{2d_i}, \text{ and so } f\in I^n,$$
a contradiction, and hence $(x_ty_t)^{m_t} \notin \{g_1,\ldots,g_n\}$.

Let $V_1 = \{k\mid  (x_kx_t)^{d_k} \in \{f_1,\ldots,f_n\}\}\setminus\{i\}$ and $V_2 = [t]\setminus (V_1\cup \{i,t\})$. Then, by Claims $1$, $3$, $4$ and $6$, we have
\begin{equation}\label{FX}
fx_i =  g (x_ix_t)^{t_id_i}\prod_{k\in V_1} (x_kx_t)^{t_kd_k} \prod_{p\in V_2\cup\{t\}} (x_py_p)^{s_p m_p},
\end{equation}
and
\begin{equation}\label{FY}
fy_i =  h(x_iy_i)^{\alpha_im_i}\prod_{a\in V_1} (x_ay_a)^{\alpha_a m_a}\prod_{b\in V_2} (x_bx_t)^{\beta_bd_b},
\end{equation}
where $t_k\geqslant 1$ for every $k\in V_1\cup \{i\}$,  $s_p\geqslant 0$ for every $p\in V_2\cup \{t\}$, $\alpha_i\ge 1$, $\alpha_a\ge 0$ for every $a\in V_1$ and $\beta_b\geqslant 0$ for every $b\in V_2$,
and
$$t_i+\sum_{k\in V_1} t_k + \sum_{p\in V_2\cup\{t\}} s_p = \alpha_i+\sum_{a\in V_1} \alpha_a+\sum_{b\in V_2}\beta_b = n.$$

Since $x_i \nmid g$, from Eqs. $(\ref{FX})$ and $(\ref{FY})$ we get $$t_id_i-1 =\deg_{x_i}(fx_i)-1 =  \deg_{x_i}(f) = \deg_{x_i}(fy_i)\geqslant \alpha_im_i,$$
and hence
\begin{equation}\label{STI}
t_id_i\geqslant \alpha_im_i+1.
\end{equation}

Next, we will show that $V_2\ne \emptyset$. Indeed, if $V_2=\emptyset$, then  $x_t^{t_id_i} \mid f$ by Eq, $(\ref{FX})$, so that $x_t^{t_id_i}\mid fy_i$. It follows that $x_t^{t_id_i}\mid h$ by Eq. $(\ref{FY})$. In particular, $x_t^{\alpha_im_i}\mid h$ by Eq. $(\ref{STI})$. Hence $h =\lambda x_t^{\alpha_im_i}$ for some monomial $\lambda$. Since $m_i> d_i$, from Eq. $(\ref{FY})$ we have
$$fy_i =  h(x_iy_i)^{\alpha_im_i}\prod_{a\in V_1} (x_ay_a)^{\alpha_a m_a},
$$
and hence
$$f = (\lambda x_i^{\alpha_im_i-\alpha_id_i}y_i^{\alpha_im_i-1})x_t^{\alpha_im_i-\alpha_id_i} (x_ix_t)^{\alpha_id_i}\prod_{a\in V_1} (x_ay_a)^{\alpha_a m_a} \in I^n,
$$
a contradiction. It follows that $V_2\ne \emptyset$.  Therefore,  there exists some $q\in V_2$. We   consider the following two subcases:

{\it Subcase $1.1$}: If $d_q < d_i$, then $s_q\geqslant 1$. Indeed, if $s_q = 0$, then  $x_q^{d_q}\mid fy_i$ by Eq. (\ref{FY}). This implies that $x_q^{d_q}\mid f$. It follows that $x_q^{d_q}\mid g$ by Eq. (\ref{FX}). Let $g = \lambda x_q^{d_q}$ for some monomial $\lambda$, then it follows from Eq. $(\ref{FX})$ that
$$f =  \lambda x_i^{d_i-1}x_t^{d_i-d_q}  (x_ix_t)^{(t_i-1)d_i} (x_qx_t)^{d_q}\prod_{x_k\in V_1} (x_kx_t)^{t_kd_k} \prod_{x_p\in V_2\cup\{t\}} (x_py_p)^{s_p m_p}\in I^n,$$
a contradiction. Hence  $s_q\geqslant 1$.

Now let $m = \lceil d_i/(d_i - d_q)\rceil$, then $m_q\geqslant md_q$ by the assumption that condition $(2)$$(a)$.  By the definition of $m$, $m \geqslant d_i/(d_i - d_q)$, i.e., $(m-1)d_i \geqslant md_q$. Hence $(x_qx_t)^{md_q} \mid x_q^{m_q}x_t^{(m-1)d_i}$, it follows that 
\begin{equation}\label{Sq}(x_qy_q)^{m_q}(x_ix_t)^{(m-1)d_i} = \lambda (x_qx_t)^{md_q}
\end{equation} for some monomial $\lambda$.

Claim:  $t_i\geqslant m-1$, where  $t_i$ is from Eq. $(\ref{FX})$. Indeed, 
$t_i \geqslant \alpha_i(m_i/d_i)+1/d_i$ by Eq. $(\ref{STI})$. This gives $t_i \geqslant \alpha_i(\lceil d_i/(d_i-d_q) \rceil-2) +1/d_i$ by assuming that condition $(2)$$(b)$. It follows that
$t_i \geqslant \alpha_i(\lceil d_i/(d_i-d_q) \rceil-2)+1=\alpha_i(m-2)+1$, since $t_i$ is an integer. In particular, $t_i \geqslant (m-2)+1=m-1$. 

From Eqs. (\ref{FX}) and (\ref{Sq}),  we get
\begin{align*}
fx_i &=  g (x_ix_t)^{t_id_i}(x_qy_q)^{s_q m_q}\prod_{k\in V_1} (x_kx_t)^{t_kd_k} \prod_{p\in (V_2 \cup\{t\})\setminus \{q\}} (x_py_p)^{s_p m_p},\\
&=g\lambda (x_qx_t)^{md_q} (x_ix_t)^{(t_i-m+1)d_i}(x_qy_q)^{(s_q-1) m_q}\prod_{k\in V_1} (x_kx_t)^{t_kd_k} \prod_{p\in (V_2 \cup\{t\})\setminus \{q\}} (x_py_p)^{s_p m_p},
\end{align*}
which contradicts Claim $1$, since $(x_qx_t)^{d_q}\in \{g_1,\ldots,g_n\}$ by Eq. (\ref{FY}). This implies that $f\in I^n$.

\medskip
{\it Subcase $1.2$}: If $d_i < d_p$ for all $p\in V_2$. Let  $d_p = \min\{d_k\mid k\in V_2\}$,  $m = \lceil d_p/(d_p - d_i)\rceil$  and $\beta= \sum_{b\in V_2} \beta_b$, where each $\beta_b$ is from Eq. $(\ref{FY})$. We distinguish into the following two cases:

(i) If $\beta\geqslant m-1$, then  by Claim $4$, we can assume that
\begin{equation}\label{NeyFY}
fy_i= hg_1\cdots g_{n-1}(x_iy_i)^{m_i},
\end{equation}
where $g_{n-m+1},\ldots,g_{n-1}\in \{(x_kx_t)^{d_k}\mid k\in V_2\}$.

Note that $\deg_{x_t}(g_{n-m+1}\cdots g_{n-1}) \geqslant (m-1)d_p$, since $\deg_{x_t}(g_k) \geqslant d_p$ for every $k\in V_2$. We can write
$g_{n-m+1}\cdots g_{n-1}=\lambda x_t^{(m-1)d_p}$ for some monomial $\lambda$.
By assuming that condition $(2)$$(a)$, we can verify that $m_i\geqslant md_i$ and $(m-1)d_p\geqslant md_i$ by the definition of $m$. It follows that
$$fy_i=(\lambda x_t^{(m-1)d_p-md_i}  x_i^{m_i-md_i}y_i^{m_i}) g_1\cdots g_{n-m} (x_ix_t)^{m d_i},$$
and hence
$$f=(\lambda x_t^{(m-1)d_p-md_i}  x_i^{m_i-md_i}y_i^{m_i-1}) g_1\cdots g_{n-m} (x_ix_t)^{m d_i} \in I^n.$$

(ii) If $\beta< m-1$, then  $x_t^{t_id_i} \mid h\prod_{b\in V_2} (x_bx_t)^{\beta_bd_b}$ by Eqs. $(7)$ and $(8)$. Hence
$h\prod_{b\in V_2} (x_bx_t)^{\beta_bd_b} = \lambda x_t^{t_id_i}$, where $\lambda$ is a monomial. 
We also write Eq. $(\ref{FY})$ in the form
$$fy_i = \lambda g_1\cdots g_{n-\beta-1} x_t^{t_id_i}(x_iy_i)^{m_i}.$$
where $g_1,\ldots,g_{n-\beta-1}$ are monomial generators of $I$. 

Note that $t_i \geqslant \alpha_i(m_i/d_i)+1/d_i \geqslant \alpha_i m + 1/d_i$ by Eq. $(\ref{STI})$ and the condition  $(2)$$(a)$. It follows that  $t_i \geqslant \alpha_i m + 1$, since $t_i$ is an integer. In particular, $t_i> m$. It follows that
$$fy_i = (\lambda x_t^{d_i(t_i-m)}y_i^{m_i})x_i^{m_i-md_i} g_1  \cdots g_{n-\beta-1} (x_ix_t)^{md_i},$$
and so
$$f = (\lambda x_t^{d_i(t_i-m)}y_i^{m_i-1})x_i^{m_i-md_i} g_1  \cdots g_{n-\beta-1} (x_ix_t)^{md_i} \in I^{n-\beta-1+m} \subseteq I^n,$$
as $\beta< m-1$.

\medskip
{\it Case $2$}: If $i = t$, then  $(x_ty_t)^{m_t}\in \{g_1,\ldots,g_n\}$ by Claim $4$. Using the same argument as at the beginning of the proof of Case $1$, we get $(x_ky_k)^{m_k} \notin \{g_1,\ldots,g_n\}$ for any $k\ne t$.
Note that $fx_t = gf_1\ldots f_n$ and $x_t\nmid g$, we have $x_t\mid f_p$ for some $p\in [n]$. Hence $f_p = (x_ty_t)^{m_t}$ or $f_p = (x_vx_t)^{d_v}$ for some $v\ne t$. The first case cannot occur by Claim $1$. So  $(x_vx_t)^{d_v}\in \{f_1,\ldots,f_n\}$.

Let $V_1 = \{k\mid (x_kx_t)^{d_k} \in \{g_1,\ldots,g_n\}\}$ and $V_2 = [t]\setminus (V_1\cup \{t\})$. Then, by Claim $1$, $3$, $4$ and $6$, we have
$$fx_t =  g\prod_{p\in V_1} (x_py_p)^{s_p m_p}\prod_{k\in V_2} (x_kx_t)^{t_kd_k},$$
and
$$fy_t =  h(x_ty_t)^{\alpha_tm_t}  \prod_{a\in V_1} (x_ax_t)^{\alpha_ad_a},$$
where  $s_p\geqslant 0$ for every $p\in V_1$, $t_k\geqslant 1$ for every $k\in V_2$,   $\alpha_t\ge 1$, $\alpha_a\ge 0$ for every $a\in V_1$,
 and
$$\sum_{p\in V_1} s_p +\sum_{k\in V_2} t_k = \alpha_t+\sum_{a\in V_1} \alpha_a= n.$$
Note that since $x_v^{d_v}\mid fx_t$ and $v\ne t$, we have $x_v^{d_v}\mid f$. This implies  that $x_v^{d_v}\mid fy_t$, so $x_v^{d_v} \mid h$, since $v \in V_2$. Let $h = \lambda x_v^{d_v}$, where  $\lambda$ is a  monomial. Since $m_t> d_v$ by the condition $(1)$ of the lemma,  it follows that
$$f = (\lambda x_t^{m_t-d_v}y_t^{m_t-1}) (x_vx_t)^{d_v}  (x_ty_t)^{(\alpha_t-1)m_t} \prod_{a\in V_1} (x_ax_t)^{\alpha_ad_a}\in I^n.$$

\medskip

In summary, we always get that  $f\in I^n$ in the expression $\p = I^n \colon f$  of the associated prime ideal $\p$, a contradiction.
The proof is now complete.
\end{proof}

We are now ready to prove Conjecture \ref{conjTrees} in the case where $G[x_1,\ldots,x_t]$ is a star.

\begin{thm} \label{thmStar}  Suppose $G[x_1,\ldots,x_t]$ is a star with the center $x_t$. Let $m_i =\omega(x_iy_i)$ for $i=1,\ldots,t$ and let $d_i = \omega(x_ix_t)$ for $i\ne t$.  Then $I(G_\omega)^n$ is Cohen-Macaulay for all $n\geqslant 1$ if and only if 
\begin{enumerate}
    \item  $m_t \geqslant 2\max\{d_i \mid i\ne t\}$.
    \item $d_i \ne d_k$ for any two different $x_ix_t,x_tx_k\in E(G)$. Furthermore, if $d_i < d_k$, then
\begin{enumerate}
\item $m_i \geqslant d_i\lceil d_k/(d_k-d_i)\rceil$, and
\item $m_k \geqslant d_k\max\{2,\lceil d_k/(d_k-d_i)\rceil-2\}$.
\end{enumerate}
\end{enumerate}
\end{thm}
\begin{proof} $(\Longrightarrow)$: This follows from Proposition \ref{CMn}.

$(\Longleftarrow)$:  We will prove that $I(G_\omega)^n$ is Cohen-Macaulay for any $n\geqslant 1$. For $k=0,\ldots,t$, we define the weight function $\omega_k$ on $G$ as follows:
$$
\omega_k(e)=
\begin{cases}
2\omega(e) & \text{ if } e = x_iy_i \text{ with } i \leqslant k,\\
\omega(e) & \text{ otherwise}.
\end{cases}
$$
In particular, $\omega_0 =\omega$. By Lemma \ref{unmixed-star-center}, we have $I(G_{\omega_k})^n$ is unmixed for all $k=0,\ldots,t$.

For $k=0,\ldots,t$, let $I_k = I(G_{\omega_k})^n$ and $J_k = (I_k)_{W_k}$ be the localization of $I_k$ with respect to $W_k$, where $W_k = \{y_1,\ldots,y_k\}$, then $J_0 = I_0 = I(G_\omega)^n$ and $J_k$ is unmixed  by Lemma \ref{localization}, since $I_k$ is unmixed. 

We now prove that $J_k$ is Cohen-Macaulay for every $k$ by the descending induction on  $k$.  If $k = t$, then
$$J_k = ((x_ix_t)^{d_i}, x_q^{2m_q}\mid i\in [t]\setminus \{t\} \text{ and } q\in [t])^n$$
is a primary monomial ideal, so it is Cohen-Macaulay. 

Suppose that $J_k$ is Cohen-Macaulay for $k > 0$. Note that 
$$(R/J_{k-1})/(y_k-x_k) \cong (R/J_k)/(y_k-x_k)$$
and $y_1-x_1,\ldots,y_t-x_t$ is a system of parameters for both $R/J_{k-1}$ and $R/J_k$. This implies that  $(R/J_k)/(y_k-x_k)$ is Cohen-Macaulay. Hence  $(R/J_{k-1})/(y_k-x_k)$ is  Cohen-Macaulay.

Next, we prove that $y_k-x_k$ is regular on $R/J_{k-1}$. So $R/J_{k-1}$ is Cohen-Macaulay.
By Lemma \ref{cmCover}, $y_k-x_k$ does not belong to any associated prime ideal of $I(G)$, so
$$y_k-x_k \notin \bigcup_{\p\in \ass(R/I(G))}\p.$$
Since  $\ass(R/J_{k-1}) \subseteq \ass( R/I(G))$, we have 
$$y_k-x_k \notin \bigcup_{\p\in \ass(R/J_{k-1})}\p.$$
Hence  $y_k-x_k$ is regular on $R/J_{k-1}$, and this implies that $R/J_{k-1}$ is Cohen-Macaulay. 

If we choose $k=0$, then  $J_0$ is Cohen-Macaulay, i.e. $I(G_\omega)^n$ is Cohen-Macaulay. The proof is now complete.
\end{proof}

\begin{rem} For the edge ideal $I$ of a simple graph or more generalized, of a weighted oriented graph (see e.g. \cite{HLMRV} for the definition of such a graph), if $I^n$ is unmixed for all $n\geqslant 1$, then the underlying graph is bipartite (see \cite{GMV, SVV}). However, this is not true for  edge-weighted graphs, as in the following result.
\end{rem}

\begin{prop} \label{triangle} Let $G$ be a connected  graph with a perfect matching $\{x_1y_1, \ldots$, $x_ty_t\}$ with $t\geqslant 2$, where each $y_i$ is a leaf vertex in $G$. Suppose  $G[x_1,\ldots,x_t]$ is a complete graph and $\omega$ is a weight  function on $E(G)$ such that
$\omega(x_jx_k)=1$ for all $x_jx_k\in E(G)$. Then $I(G_\omega)^n$ is Cohen-Macaulay for all $n\geqslant 1$ if and only if $\omega(x_iy_i)\geqslant 2$ for all $i=1,\ldots,t$.
\end{prop}
\begin{proof}  $(\Longrightarrow)$: For each $i$, we choose  $j\ne i$. Let $H=G\setminus\{x_s\mid s\notin \{i,j\}\}$, then $E(H)=\{x_iy_i,x_ix_j,x_jy_j\}$. Since $I(G_\omega)^n$ is Cohen-Macaulay for all $n\geqslant 1$, $I(H_\omega)^n$ is also Cohen-Macaulay. Thus $\omega(x_iy_i)\geqslant 2$ by Proposition \ref{n2}.

$(\Longleftarrow)$: We will prove that $I(G_\omega)^n$ is Cohen-Macaulay by induction on $n$. The  case $n \leqslant 2$ is shown separately in Theorem \ref{Tk}. So we can assume that  $n\geqslant 3$, and prove the statements by  induction on $t$. The case  $t \leqslant 2$ is shown separately
in  Proposition \ref{n2}. Thus, in the following, we can assume that  $t\geqslant 3$. We will divide the proof into the following two steps:
\medskip

{\it Step $1$}: We prove that $I(G_\omega)^n$ is unmixed. Let $I = I(G_\omega)$ and  $m_i = \omega(x_iy_i)$ for $i=1,\ldots,t$. Conversely, suppose that $I^n$ is not unmixed, then there exists some monomial $f\notin I^n$  such that $fx_i,f y_i\in I^n$ for some  $i\in [t]$. 
We can write $f x_i = gf_1\cdots f_n \text{ and } f y_i= h g_1\cdots g_n$, where $g,h$ are monomials and $f_i,g_i$ are monomial generators of $I$ for $i=1,\ldots,n$. Similarly to Claims  $1$, $4$ and $5$  of the proof in  Lemma \ref{unmixed-star-center},  we 
obtain that
\begin{enumerate}
    \item $\{f_1,\ldots,f_n\} \cap \{g_1,\ldots,g_n\}=\emptyset$;
    \item there exists some $p\in [n]$ such that  $g_p = (x_iy_i)^{m_i}$;
   \item  there exists some $p\in [n]$ such that $f_p =x_ix_j$ for some $j\ne i$.  
\end{enumerate}
We  can assume that $g_n = (x_iy_i)^{m_i}$ and $f_n = x_ix_j$ for some $j\ne i$.
{\it Claim}: $(x_ky_k)^{m_k}\notin \{g_1,\ldots,g_n\}$ for every $k\ne i$. In fact,  assume that $(x_ky_k)^{m_k}\in \{g_1,\ldots,g_n\}$ for some $k\ne i$. Then, we can assume that $g_{n-1}=(x_ky_k)^{m_k}$, so that
$$fy_i = hg_1\cdots g_{n-2}(x_ky_k)^{m_k}(x_iy_i)^{m_i}.$$
Therefore,
$$f = (h x_k^{m_k-2} y_k^{m_k}x_i^{m_i-2} y_i^{m_i-1})g_1\cdots g_{n-2}(x_ix_k)^2\in I^n,$$
a contradiction, and the claim follows.

\medskip
Since $x_ix_j\mid fx_i$, we have $x_j\mid f$, and hence  $x_j\mid f y_i$. It follows that either $x_j\mid h$ or $x_j\mid g_k$ for some $k\ne n$. In the former case
$f = (h/x_j)(x_i^{m_i-1}y_i^{m_i-1})g_1\cdots g_{n-1}(x_ix_j) \in I^n$, a contradiction. In the second case, by Claim, (1) and (3), we can assume that $g_k = x_jx_q$ for some $q\notin \{i,j\}$. Without loss of generality, we can assume that $k=n-1$. Then
$$f = (hx_i^{m_i-2}y_i^{m_i-1}) g_1\cdots g_{n-2} (x_ix_j)(x_ix_q) \in I^n,$$
a contradiction. This implies that $I^n$ is unmixed.
\medskip

{\it Step $2$}: We prove that $I(G_\omega)^n$ is Cohen-Macaulay. This can be shown by similar arguments as in the proof of Theorem \ref{thmStar}, so
we omit its proof.
\end{proof}


\medskip
\hspace{-6mm} {\bf Data availability}

The data that has been used is conﬁdential.

\medskip
\hspace{-6mm} {\bf Acknowledgement}

 \vspace{3mm}
\hspace{-6mm}  The third author is  supported by the Natural Science Foundation of Jiangsu Province (No. BK20221353) and
the National Natural Science Foundation of China (12471246). Part of this work was done while the second and the third authors were at the Vietnam Institute of Advanced Studies in Mathematics (VIASM) in Hanoi, Vietnam. We would like to thank VIASM for its hospitality.

\end{document}